\documentclass[12pt]{amsart}
\usepackage[headings]{fullpage}
\usepackage{amssymb,epic,eepic,epsfig,amsbsy,amsmath,amscd,color,svg}
\usepackage[all]{xy}
\usepackage{graphicx}
\usepackage{texdraw,tikz-cd}
\usepackage{url}
\usepackage{bbm}
\usepackage{mathrsfs}
\usepackage{scalerel}

\usepackage{amsmath,amssymb,amsthm,mathtools}
\usepackage[bookmarks=true,colorlinks,linkcolor=blue]{hyperref}
\usepackage{graphicx,subcaption}
\usepackage{standalone,tikz,tikz-cd}
\usetikzlibrary{positioning,knots,patterns,hobby,decorations.pathreplacing}
\usepackage{comment}
\usepackage{accents}
\usepackage{extarrows}

\newcommand\no[1]{}

\usepackage{float}   
\usepackage{mathrsfs}

\theoremstyle{plain}
\newtheorem{theorem}{Theorem}[section]
\newtheorem*{theorem*}{Theorem}
\newtheorem{lemma}[theorem]{Lemma}
\newtheorem{corollary}[theorem]{Corollary}
\newtheorem{proposition}[theorem]{Proposition}
\newtheorem{conjecture}[theorem]{Conjecture}
\newtheorem{question}[theorem]{Question}

\newtheorem{definition}[theorem]{Definition}
\newtheorem{thm}{Theorem}

\theoremstyle{definition}
\newtheorem{remark}[theorem]{Remark}
\newtheorem{example}[theorem]{Example}

\newcommand{\bcon}{\begin{conjecture}}
\newcommand{\econ}{\end{conjecture}}
\newcommand{\bcor}{\begin{corollary}}
\newcommand{\ecor}{\end{corollary}}
\newcommand{\bdf}{\begin{definition}}
\newcommand{\edf}{\end{definition}}

\newcommand{\benu}{\begin{enumerate}}
\newcommand{\eenu}{\end{enumerate}}
\newcommand{\beq}{ \begin{equation} }
\newcommand{\eeq}{ \end{equation}}
\def\be { \begin{equation} }
\def\ee { \end{equation} }
\newcommand{\bexa}{\begin{example}}
\newcommand{\eexa}{\end{example}}
\newcommand{\bexe}{\begin{exercise}}
\newcommand{\eexe}{\end{exercise}}
\newcommand{\bfac}{\begin{fact}}
\newcommand{\efac}{\end{fact}}
\newcommand{\bite}{\begin{itemize}}
\newcommand{\eite}{\end{itemize}}
\newcommand{\blem}{\begin{lemma}}
\newcommand{\elem}{\end{lemma}}
\newcommand{\bmat}{\begin{pmatrix}}
\newcommand{\emat}{\end{pmatrix}}
\newcommand{\bprb}{\begin{problem}}
\newcommand{\eprb}{\end{problem}}
\newcommand{\bpro}{\begin{proposition}}
\newcommand{\epro}{\end{proposition}}

\newcommand{\bque}{\begin{question}}
\newcommand{\eque}{\end{question}}
\newcommand{\brem}{\begin{remark}}
\newcommand{\erem}{\end{remark}}
\newcommand{\bthm}{\begin{theorem}}
\newcommand{\ethm}{\end{theorem}}
\newcommand{\bpr}{\begin{proof}}
\newcommand{\epr}{\end{proof}}

\usepackage[T1]{fontenc}
\makeatletter
\newcommand*{\doublerightarrow}[2]{\mathrel{
  \settowidth{\@tempdima}{$\scriptstyle#1$}
  \settowidth{\@tempdimb}{$\scriptstyle#2$}
  \ifdim\@tempdimb>\@tempdima \@tempdima=\@tempdimb\fi
  \mathop{\vcenter{
    \offinterlineskip\ialign{\hbox to\dimexpr\@tempdima+1em{##}\cr
    \rightarrowfill\cr\noalign{\kern.5ex}
    \rightarrowfill\cr}}}\limits^{\!#1}_{\!#2}}}
\newcommand*{\triplerightarrow}[1]{\mathrel{
  \settowidth{\@tempdima}{$\scriptstyle#1$}
  \mathop{\vcenter{
    \offinterlineskip\ialign{\hbox to\dimexpr\@tempdima+1em{##}\cr
    \rightarrowfill\cr\noalign{\kern.5ex}
    \rightarrowfill\cr\noalign{\kern.5ex}
    \rightarrowfill\cr}}}\limits^{\!#1}}}
\makeatother

\def\BC{\mathbb C}
\def\BN{\mathbb N}
\def\BZ{\mathbb Z}

\def\BM{\mathbb M}
\def\BN{\mathbb N}

\def\Dcob{\mathrm{DeCob}}

\def\Bim{\mathrm{Mor}}

\def\Cut{\mathsf{Cut}}

\def\2Mor{\mathrm{2Mor}}

\def\OSL{{\mathcal O}_{q^2}(\mathrm{SL}(2))}

\def\bSi{\overline \Sigma}

\def\cS{\mathscr S}
\def\ot{\otimes}

\def\cE{\mathcal E}
\def\cF{\mathcal F}

\def\bk{\mathbf k}

\def\cP{\mathcal P}

\def\Id{\mathrm{Id}}
\def\fS{\mathfrak S}
\def\bfS{\overline{\fS}}

\def\D{\Delta}

\def\cY{\mathcal Y}
\def\ev{{\mathrm{ev}}}

\def\embed{\hookrightarrow}

\def\bbS{\overline \Sigma}

\def\cY{\mathcal Y}

\def\ev{{\mathrm{ev}}}

\def\embed{\hookrightarrow}

\def\SM{\cS_{q^{1/2}}(M)}
\def\pS{\partial \Sigma}

\def\Ss{\cSs}

\def\Tr{\mathrm{Tr}}

\DeclareMathOperator{\tr}{\mathrm tr}

\def\al{\alpha}

\def\ve{\varepsilon}
\def\be { \begin{equation} }
\def\ee { \end{equation} }

\def\P{\mathcal P}

\def\nc{\newcommand}

\nc\FIGjpg[2]{\begin{figure}
    \includegraphics[width=\textwidth]{#1.jpg} 
   \caption{#2}
    \label{fig:#1}
    \end{figure}}
\nc\FIGc[3]{\begin{figure}[htpb]
    \includegraphics[height=#3]{#1-eps-converted-to.pdf}
    \caption{#2}
    \label{fig:#1}
    \end{figure}}
    
     \nc\FIGceps[3]{\begin{figure}[htpb]
    \includegraphics[height=#3]{#1-eps-converted-to.pdf}
    \caption{#2}
    \label{fig:#1}
    \end{figure}}   
    
\newcommand\incl[2]{{\includegraphics[height=#1]{#2-eps-converted-to.pdf}}}
\newcommand\incleps[2]{{\includegraphics[height=#1]{#2-eps-converted-to.pdf}}}

\def\leftve{\raisebox{-8pt}{\incleps{.8 cm}{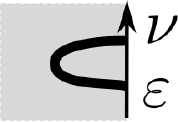}}}
\def\leftvebar{\raisebox{-8pt}{\incleps{.8 cm}{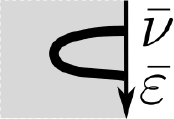}}}
\def\leftved{\raisebox{-8pt}{\incleps{.8 cm}{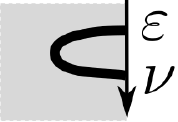}}}

\def\emptyr{\raisebox{-8pt}{\incl{.8 cm}{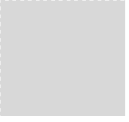}}}
\def\emptys{\raisebox{-8pt}{\incl{.8 cm}{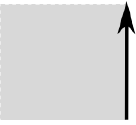}}}
\def\emptyd{\raisebox{-8pt}{\incleps{.8 cm}{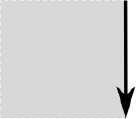}}}

\def\cross{  \raisebox{-8pt}{\incl{.8 cm}{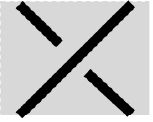}} }
\def\resoP{  \raisebox{-8pt}{\incl{.8 cm}{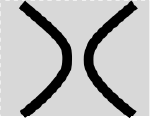}} }
\def\resoN{  \raisebox{-8pt}{\incl{.8 cm}{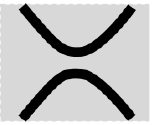}} }
\def\kinkp{  \raisebox{-8pt}{\incleps{.8 cm}{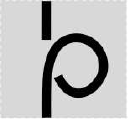}} }
\def\kinkn{  \raisebox{-8pt}{\incleps{.8 cm}{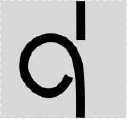}} }
\def\kinkzero{  \raisebox{-8pt}{\incleps{.8 cm}{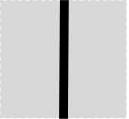}} }

\def\reordonepn{  \raisebox{-8pt}{\incl{.8 cm}{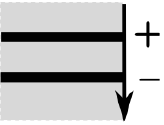}} }
\def\reordtwopn{  \raisebox{-8pt}{\incl{.8 cm}{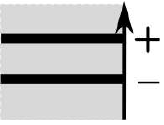}} }

\def\reordnp{  \raisebox{-8pt}{\incl{.8 cm}{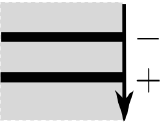}} }

\def\reordpn{  \raisebox{-8pt}{\incl{.8 cm}{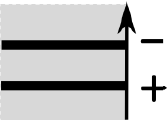}} }
\def\reordone{  \raisebox{-8pt}{\incl{.8 cm}{reord1}} }
\def\reordtwo{  \raisebox{-8pt}{\incl{.8 cm}{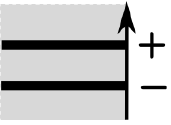}} }
\def\reordthree{  \raisebox{-8pt}{\incl{.8 cm}{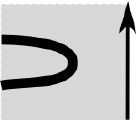}} }
\def\trivloop{  \raisebox{-8pt}{\incl{.8 cm}{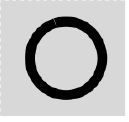}} }

\def\sign{\mathsf{Sign}}

\begin{document}

\title{Stated skein modules of 3-manifolds and TQFT}
\author[Francesco Costantino]{Francesco Costantino}
\address{Institut de Math\'ematiques de Toulouse, 118 route de Narbonne,
F-31062 Toulouse, France}
\email{francesco.costantino@math.univ-toulouse.fr}

\author[Thang  T. Q. L\^e]{Thang  T. Q. L\^e}
\address{School of Mathematics, 686 Cherry Street,
 Georgia Tech, Atlanta, GA 30332, USA}
\email{letu@math.gatech.edu}

\date{}

\thanks{
2010 {\em Mathematics Classification:} Primary 57N10. Secondary 57M25.}

\begin{abstract}We study the behaviour of the Kauffman bracket skein modules of 3-manifolds under gluing along surfaces. For this purpose we extend the notion of Kauffman bracket  skein modules to $3$-manifolds with marking consisting of open intervals and circles in the boundary. The new module is called the {\em stated skein module}.

The first main results concern non-injectivity of certain natural maps defined when forming connected sums along a sphere or along a closed disk. These maps are injective for surfaces, or for generic quantum parameter, but we show that in general they are not injective when the quantum parameter is a root of 1. The result applies to the classical skein modules as well.
A particular interesting result is that when the quantum parameter is a root of 1, the empty skein is zero  in a connected sum where each constituent manifold  has non-empty marking. 
 We also prove various non injectivity results  for the Chebyshev-Frobenius map and the natural map induced by the  deletion of marked balls. 
 
We then consider the general case of gluing along a surface,
showing that the stated skein module can be interpreted as a monoidal symmetric functor from a category of ``decorated cobordisms'' to a Morita category of algebras and their bimodules. We apply this result to deduce several properties of stated skein modules as a Van-Kampen like theorem as well as a computation through Heegaard decompositions and a relation to Hochshild homology for trivial circle bundles over surfaces.

\end{abstract}

\maketitle

\def\pbbS{\partial \bbS}
\def\bSP{(\bbS,\cP)}
\def\cA{\mathcal A}
\def\cB{\mathcal B}
\def\Si{\Sigma}
\def\fB{\mathfrak B}
\def\pr{\mathrm{pr}}
\def\cO{\mathcal O}

\def\ord{\mathrm{ord}}
\def\Sqq{\cS_{q^{1/2}}\MN}
\def\Sq{\cS_{q^{1/2}}}

\def\poS{\partial_0\Sigma}
\def\cSs{\cS}
\def\RcSs{\cS^{rel}_{\mathrm s}}
\def\zRcSs{\cS^{rel}_{\mathrm s,0}}

\def\cR{{R}}
\def\basics{basic skein}
\def\YD{\cY(\D)}
\def\tYD{\tilde \cY(\D)}
\def\YtD{\cY^{(2)}(\D)}
\def\tYtD{\tilde \cY^{(2)}(\D)}
\def\bve{{\boldsymbol{\ve}}}
\def\bm{{\mathbf m}}
\def\hYeD{\tilde \cY^\ev(\D)}
\def\YeD{\cY^\ev(\D)}
\def\pfS{\partial \fS}
\def\pbfS{\partial \bfS}
\def\bove{\boldsymbol{\epsilon}}
\def\bomu{{\vec \mu}}
\def\bonu{{\vec \nu}}
\def\Gr{\mathrm{Gr}}
\def\bl{\mathbf{l}}
\def\onto{\twoheadrightarrow}
\def\Stink{\mathrm{St}^\uparrow(\bk)}
\def\Stinkp{\mathrm{St}^\uparrow(\bk')}
\def\sincr{s^\uparrow}
\def\St{\mathrm{St}}

\def\tF{\tilde {\cF}}
\def\tE{\tilde {\cE}}
\def\tfT{\tilde {\fT}}
\def\bcSs{\overline{\cS}_s}
\def\cN{\mathcal N}
\def\MN{(M,\cN)}
\def\pM{\partial M}
 \def\cSsp{\cS_{+}}
   \def\TrD{\Tr_\D}
   \def\hTrD{\widehat{\Tr}_\D}
   \def\hcSs{\widehat{\cS}_s}
   \def\SS{\cS(\fS)}
   \def\JWm{{f}}
\tableofcontents
\section{Introduction}

\subsection{Kauffman bracket skein module} The Kauffman skein module $\SM$ of an oriented 3-manifold $M$, introduced by Przytycki \cite{Prz} and Turaev \cite{Turaev0,Turaev}, serves as a bridge between quantum and classical topology, and helps solving many important problems in low-dimensional topology. By definition $\SM$   is the $\BC$-vector space spanned by isotopy classes of unoriented framed links in $M$ subject to the Kauffman relations (\cite{Kauffman})
\begin{align*}
\begin{tikzpicture}[scale=0.8,baseline=0.3cm]
\fill[gray!20!white] (-0.1,0)rectangle(1.1,1);
\begin{knot}[clip width=8,background color=gray!20!white]
\strand[very thick] (1,1)--(0,0);
\strand[very thick] (0,1)--(1,0);
\end{knot}
\end{tikzpicture}
&=q
\begin{tikzpicture}[scale=0.8,baseline=0.3cm]
\fill[gray!20!white] (-0.1,0)rectangle(1.1,1);
\draw[very thick] (0,0)..controls (0.5,0.5)..(0,1);
\draw[very thick] (1,0)..controls (0.5,0.5)..(1,1);
\end{tikzpicture}
+q^{-1}
\begin{tikzpicture}[scale=0.8,baseline=0.3cm]
\fill[gray!20!white] (-0.1,0)rectangle(1.1,1);
\draw[very thick] (0,0)..controls (0.5,0.5)..(1,0);
\draw[very thick] (0,1)..controls (0.5,0.5)..(1,1);
\end{tikzpicture}\, ,\qquad
\begin{tikzpicture}[scale=0.8,baseline=0.3cm]
\fill[gray!20!white] (0,0)rectangle(1,1);
\draw[very thick] (0.5,0.5)circle(0.3);
\end{tikzpicture}
=(-q^2 -q^{-2})
\begin{tikzpicture}[scale=0.8,baseline=0.3cm]
\fill[gray!20!white] (0,0)rectangle(1,1);
\end{tikzpicture}\, .
\end{align*}
 See Section \ref{sec.prelim} for details. Here $q^{1/2}$ is a non-zero complex number.

The calculation of $\SM$ is in general difficult. One attractive approach is to cut $M$ into simpler pieces and try to understand $\SM$ from the skein modules of the pieces. For example when $M= M_1 \# M_2$ is the connected sum of $M_1$ and $M_2$ there is a natural $\BC$-linear homomorphism
\beq
\Psi_{M_1,M_2} : \Sq(M_1)  \ot_\BC \Sq M_2 \to \SM, \quad (x \ot y) \to x \sqcup y.
\notag
\eeq
Przytycki \cite{Prz} showed that if $q$ is not a root of 1 then $\Psi_{M_1, M_2}$ is bijective. Our first result is to show that when $q$ is a root of 1 the kernel of $f_{M_1,M_2}$ is generally big.

\begin{thm}[Special case of Theorem \ref{thm.connected3}]
\label{thm.t1}
 Suppose $q^4$ is a primitive $N$-th root of 1. The kernel of $\Psi_{M_1 \# M_2} $ contains $F_{q^{1/2} } (M_1) \ot F_{q^{1/2} } (M_2)$, where $F_{q^{1/2} } (M)$ is the subspace of $\SM$ spanned by closures of the Jones-Wenzl idempotent $\JWm_{N-1}$.
\end{thm}
We recall  the Jones-Wenzl idempotent in Subsection \ref{sec.JW}. In particular when $M_i$ are thickened surfaces we show that each $F_{q^{1/2} } (M_i)$ is non-zero, hence the kernel of $\Psi_{M_1 \# M_2}$ is non-trivial. For work related to Theorem \ref{thm.t1} see Remark \ref{rem.connected}.

Suppose $q^4$ is a primitive $N$-th root of unity. Then $\ve:= q^{N^2}$ has property $\ve^8=1$. There exists a 
$\BC$-linear map, known as the Chebyshev-Frobenius map,
$$ \Phi_{q^{1/2}} : \cS_\ve(M) \to \SM$$
which was constructed by Bonahon and Wong \cite{BW}, see also \cite{Le3} especially for general 3-manifolds. For the thickened surfaces it is known that  $\Phi_{q^{1/2}}$ is injective. We show that in general $\Phi_{q^{1/2}}$ is not injective.

\begin{thm}[See Theorem \ref{thm.Frobenius}]
\label{thm.t2}
Assume $q^4$ is a primitive $N$-th root of 1 with $N >1$. There exists a compact oriented 3-manifold $M$ such that the Chebyshev-Frobenius homomorphism $ \Phi_{q^{1/2}} : \cS_\ve(M) \to \SM$ is not injective.
\end{thm}

Theorems \ref{thm.t1} and \ref{thm.t2}, as well as their analogs for the stated skein module case, exhibit the surprising fact  that at roots of 1 certain skein identities are not local as they can only be established by means of tangles far away from their supports. This is very counter-intuitive in the theory of skein modules.

\subsection{Marked 3-manifolds, stated skein modules, and non-injectivity} Suppose $S \subset M$ is a properly embedded surface. Let $M'$ be the result of cutting $M$ along $S$. 
\no{
For $i=1,2$ suppose $S_i \subset \partial M'$ is a compact surface in the boundary of a compact oriented 3-manifold $M'$, possibly disconnected, and $f: S_1 \to S_2$ is an orientation reversing diffeomorphism.  Let $M= M'/f$, the result of identifying $S_1$ with $S_2$ via $f$. Let $S\subset M$ be the common image of $S_1$ and $S_2$. We also say that $M'$ is the result of cutting $M$ along $S$.}
 The goal is to understand the skein module of $M$ though that of $M'$. For this purpose in \cite{Le:TDEC,CL,BL,LY1} we with collaborators extended the definition of skein modules of 3-manifolds to {\em stated} skein modules of {\em marked 3-manifolds}, where the marking consists of disjoint oriented interval in the boundary of $M$. 
A main result is the existence of a cutting homomorphism relating the stated skein modules of $M$ and $M'$ when $S=D^2$, the closed disk. When $M$ is a thickened surface the cutting homomorphism is always injective, by \cite[Theorem 1]{Le:TDEC}. We show that in general the cutting homomorphism is not injective.
\begin{thm}[See Theorem \ref{thm.CH}]  Suppose $q^4$ is a primitive $N$-root of 1 with $N >1$. There exists a marked 3-manifold $M$ and a properly embedded disk $D \embed M$ such that the cutting homomorphism corresponding to the cutting of $M$ along $D$ is not injective.
\end{thm}

When $q$ is not a root of 1, one can show (see Proposition \ref{prop:sphere}) that $\Sq(M_1 \# M_2)$ is spanned by skeins with support disjoint from the cutting sphere $S$ which realises the connected sum. It turns out that the picture is quite opposite at root of 1, in the presence of markings.

\begin{thm}[ Special case of Theorem \ref{thm.zero} ]Suppose $q^4$ is a primitive $N$-th root of 1 with $N >1$.
Assume that each of $M_1$ and $M_2$ is connected and has at least one marking. Then every skein in $ M_1 \# M_2$ whose support is disjoint from the cutting sphere $S$ is equal to 0. In particular the empty skein is 0.
\end{thm}

The framework of stated skeins allows to properly study glueing and cutting operations on manifolds and interpret them algebraically. This point of view had been explored in the case of stated skein algebras of surfaces in \cite{CL} (see also \cite{KQ}) where, for instance, it had been shown that the quantised coordinate algebra $\OSL$ of $SL_2(\BC)$, see Section \ref{sec.prelim}, is naturally isomorphic to the stated skein algebra of the bigon and as such it co-acts on all stated skein modules of surfaces. 

\subsection{More general marking, gluing along general surfaces}
Cutting along an embedded closed disk is understood via the cutting homomorphism. We want to consider cutting along more general surfaces and include the stated skein module into a topological quantum field theory (TQFT).

To make the theory more fluid, we will extend the stated skein modules to marked manifolds where the marking includes circles, in addition to intervals. For the details see  Section~\ref{sec.prelim}. This setting is new even for surfaces, even though in the presence of a circular marking we loose the algebra structure of stated skein modules of surfaces. However, we can do cutting along circle:
\begin{thm}[See Theorem \ref{teo:cutting} for more precise statement] Suppose $\fS'$ is the result of cutting a marked surface $\fS$ along a circle. There is a naturally defined $\BC$-linear map $\rho: \Sq(\fS) \to \Sq(\fS')$, given by an explicit state sum formula.
\end{thm}
Cutting along a circle provides more flexibility in the study of the skein modules of surfaces. 
We present a basis for these skein modules in Theorem \ref{teo:basis} and Proposition \ref{thm.slit2}, extending the previous analogous theorem of \cite{Le:TDEC}.
In particular we recover the Hochshild homology $\rm{HH}_0(\OSL)$ of $\OSL$ as the stated skein module of an annulus with two circular marked components. 

The cutting homomorphism corresponding to cutting along an embedded closed disk can be defined as in the case without circular marking, see Theorem \ref{thm.gluing}.

Besides the cutting operation we introduce the {\em slitting operation}. When cutting a surface $\fS$  along an ideal arc $c$, we get a new surface $\fS'$ having two copies $c_1, c_2$ of $c$ such that by identifying $c_1$ with $c_2$ we get back $\fS$. On the other hand, splitting $\fS$ along a properly embedded arc $d$ (not an ideal arc) means simply to remove $d$ from $\fS$. There is also an operation of slitting along a half-ideal arc, i.e. an arc one endpoint of which is an ideal point and the other endpoint is on the boundary of $\fS$. We describe how the skein modules behave under the slitting operations in Theorems \ref{thm.slit1} and \ref{thm.slit2}. The slitting operations allow even more flexibility in studying skein modules.

 The geometric setting gives additional, natural structures on the stated skein module of a marked three manifold $M$.  Each connected component of the marking   defines a comodule structure on $\SM$ over $\OSL$ or over the coalgebra $\rm{HH}_0(\OSL)$ according as the component is an interval or a circle. 
 
 \def\pSi{\partial \Sigma}
 
The slitting operation can be generalised to the following more general situation.
 Assume $\Sigma$ is a compact surface and $\cP$ a finite set of points in the boundary $\pSi$, where each point  is equipped with a sign $\pm$. The thickened surface $(\Sigma\times (-1,1), \cP\times (-1,1))$ is a marked 3-manifold, and its stated skein module, denoted by $\cS(\Sigma,\cP)$, has a natural structure of an algebra, where the product is defined by the usual stacking operation. Assume $\MN$ is a marked 3-manifold and $\Sigma\embed M$ is a compact, oriented, connected, properly embedded surface which meets $\cN$ transversally. Let $(M', \cN')$ be the result of removing $\Sigma$ from $\MN$. Then there are  natural left and right actions of the algebra $\cS(\Sigma, \Sigma\cap \cN)$ on $\cS(M', \cN')$, making the latter a bimodule over $\cS(\Sigma, \Sigma\cap \cN)$. For a bimodule $V$ over an algebra $A$ one can define the 0-th Hochschild homology by
 $$ HH_0(V) = V/ ( v * a - a *v).$$
 We prove the following: 
\begin{thm}[See Theorem \ref{thm.split3}] The inclusion $(M',\cN')\to (M,\cN)$ induces an isomorphism of $\cR$-modules: 
$${\rm HH}_0(\cSs(M',\cN'))\xrightarrow{\cong} \cSs(M,\cN).$$
\end{thm}

As shown in Examples \ref{ex:halfslitrecover} and \ref{ex:cptslitrecover}, this result encompasses multiple previous statements; it also allows one to generalise the ``triangle sum'' of surfaces  studied in \cite{CL} to the case of marked $3$-manifolds and prove in Theorem \ref{thm:3Dtrianglesum} that if $M$ is the triangle sum of $M_1$ and $M_2$ then there is a natural $\BC$-linear isomorphism $\SM\cong \Sq(M_1)\otimes \Sq(M_2)$.

We conclude the paper by defining a category of decorated cobordisms whose objects are {\em marked surfaces} and morphisms are marked $3$-manifolds whose boundary is suitably decomposed into positive, negative and ``side'' parts.
We then show that $\cSs$ can be interpreted as a functor from this category to the Morita category ${\mathsf{ Morita}}$ of algebras and their bimodules. The main result of our TQFT theory is  the following: 
\begin{thm}[See Theorem \ref{thm.tqft}]\label{thm.seven}
The stated skein functor $\cSs:\Dcob\to {\mathsf{ Morita}}$ is a symmetric monoidal functor. 
\end{thm}

Immediate consequences of Theorems \ref{thm.tqft} and \ref{thm.split3} are a Van Kampen-like theorem for stated skein modules (Theorem \ref{teo:fundgroup}), a description of $\cSs(M)$ given a Heegaard decomposition of $M$ (Theorem \ref{teo:heegaard}) and the fact that $\cSs(\fS\times S^1)={\rm HH}_0(\cSs(\fS))$ (Proposition \ref{prop:HH0}). The computation of $\cSs(M,\cN)$ from a Heegaard decomposition was already obtained in \cite{GJS}. Furthermore the Van Kampen like theorem is very much similar in spirit to Habiro's quantum fundamental group behaviour (\cite{Ha}). 
In \cite{CL3} we will show that the stated skein functor, restricted to a suitable category, is actually a braided monoidal functor with values in suitable Morita-like category of braided comodule algebras over $\OSL$. 

Remark that the TQFT described in Theorem \ref{thm.tqft} is very different from those studied for instance in \cite{RT,BHMV} as the target category is not that of vector spaces. Rather, when working over a field, it seems to fit very well in the general framework of \cite{BJSS} where in particular a semisimple cp-rigid and cocomplete braided monoidal category is shown to be a $3$-dualisable object in the $4$-category ${\bf BrTens}$ so that, as a consequence, there is an extended TQFT associated to it.
We expect that our construction is a special case of this, for the ribbon Hopf algebra of right $\OSL$-comodules at least when $q$ is a generic complex number, although there are some aspects in which our construction is more general in the sense that is allows multiple markings on manifolds and, more importantly, circular ones. 
All these connections are still to be explored as it has been done in \cite{Hai} for the relations between stated skein algebras of surfaces and integrals of ribbon categories over surfaces. 
While completing the present paper we were informed that a result similar to Theorem \ref{thm.seven} was independently proved by J. Korinman and J. Murakami in a forthcoming preprint.

It turns out (see \cite{Hai}) that over a field  stated skein algebras of surfaces are isomorphic to the algebras obtained by ``integrating'' the ribbon category of $\OSL$-comodules over the surface as explained in \cite{BBJ}. Furthermore when the surfaces have only one boundary component with a single marking, these algebras are isomorphic to those defined by Alekseev, Grosse and Schomerus (\cite{AGS}) and by Buffenoir and Roche (\cite{BR}) (see also \cite{Fai,LY1}). For 3-manifolds our stated skein modules are similar but different by nature from the skein category defined eg in  \cite{JF,Cooke}. Our approach is more elementary and geometric in nature, with explicit generators and relations. Moreover, it works over any ground ring, allowing for example to analyse the stated skein modules/algebras at roots of 1, and to find embeddings of stated skein algebras into quantum tori, rings of Laurent polynomials in $q$-commuting variables, see for example \cite{LY2}.

\subsection{Acknowledgement} 
F.C. is grateful to Utah State University in Logan where this work was concluded. He also acknowledges the funding from CIMI Labex ANR 11-LABX-0040 at IMT Toulouse within the program ANR-11-IDEX-0002-02 and from the french ANR Project CATORE ANR-18-CE40-0024. 

 T. L. is partially supported by NSF grant 1811114. 
 
The authors would like to thank David Jordan, Julien Korinman, Adam Sikora and Dominic Weiller for helpful discussions.

\section{Stated skein modules of marked 3-manifolds} \label{sec.prelim}
Throughout the paper let
$\BZ$ be the set of integers, $\BN$ be the set of non-negative integers, $\BC$ be  the set of complex numbers.
The ground ring $\cR$  is a commutative ring with unit 1,  containing a distinguished invertible element $q^{1/2}$. 

The Kronecker delta is defined as usual: \ $ \delta_{x,y} = \begin{cases} 1\quad  &\text{if} \ x=y,\\
  0 &\text{if} \ x\neq y
  \end{cases}
  $
  
 For a finite set $X$ we denote by $|X|$ the number of elements of $X$.

\subsection{Marked 3-manifold} By a {\em open  interval} (respectively {\em circle}) we mean the image of $(0,1)$ (resp. the standard circle $S^1$) through an embedding of $[0,1]$ (resp. of $S^1$)  into a manifold.

\bdf {\em A marked 3-manifold $\MN$} consists of
an oriented 3-manifold $M$  with (possibly empty) boundary $\pM$ and  a 1-dimensional oriented
submanifold $\cN \subset \pM$ such that $\cN$ is the disjoint union of several {open intervals and circles}; we will refer to the intervals as ``boundary edges'' or ``edges'' and to the circles as ``marked circles''.

An embedding of pairs of marked 3-manifolds $i:(M,\cN)\hookrightarrow (M',\cN')$ is an orientation preserving proper embedding $i:M\to M'$ such that $i(\cN)\subset \cN'$ and $i$ preserves the orientation on $\cN$. 
\edf
 A priori two  components of $\cN$ might be mapped by $i$ into the same component of $\cN'$. 
 
 If no component of $\cN$ is a circle, we call $\MN$ a {\em circle-free} marked 3-manifold.

\brem Our notion of a marked 3-manifold is more general than that in \cite{LY1, BL} where only circle-free marked 3-manifolds are considered.
\erem

\begin{definition} Let $\MN$ be a marked 3-manifold. {\em An $\cN$-tangle $L$} (in $M$) is a 1-dimensional, compact,  non-oriented smooth submanifold of $M$ equipped
with a normal vector field
such that $L \cap
\cN = \partial L$  and at a boundary point in $\partial L=L \cap
\cN $, the normal vector is a positive tangent of $\cN$.

Here a {\em normal
vector field} on $L$  is a vector field not tangent to $L$ at any point. 

A loop component of $L$, i.e. a component diffeomorphic to $S^1$, is called a $\cN$-knot, and a non-loop component, which must be  diffeomorphic to $[0,1]$, is called an $\cN$-arc.

Two  $\cN$-tangle are {\em $\cN$-isotopic} if they are
 isotopic through the class of  $\cN$-tangles.
 
 The empty set is also considered a  $\cN$-tangle which is isotopic only to itself.
 
 A state of an $\cN$-tangle $L$  is a map $s:\partial L\to \{\pm\}$. The switching map  $\{ \pm \} \to \{ \pm \} $ is the involution $\epsilon \to \bar \epsilon: =-\epsilon$. The set $\{\pm \}$ is order so that $ -$ is smaller than $+$. A state is \emph{increasing} if while moving along any boundary edge in the positive direction, i.e. the direction induced from the orientatin of the surface, the state function is increasing, i.e. one encounters 
 first a sequence of $-$ and then a sequence of $+$.
\end{definition}
 
It should be noted that while $M,\cN$ are oriented, an $\cN$-tangle is not.

\bdf
The {\em stated  skein module} $\cSs(M,\cN)$ of a marked 3-manifold $\MN$ is the $\cR$-module spanned by isotopy classes of stated $\cN$-tangles in $M$ modulo the following relations:
\begin{align}
\label{eq.skein0} \cross \ &= \ q\resoP + q^{-1} \resoN\\
\label{eq.loop0}  \trivloop\  &=\  (-q^2 -q^{-2})\emptyr\\
\label{eq.arcs0} \leftve\  & =\  \delta_{\epsilon, \bar \nu} C(\epsilon)   \emptys\ , \text{where}\ C(+)=  -q^{-\frac 52} ,  C(-)=  q^{-\frac 12}  \\
 \label{eq.order0}   \reordone\ &=\  q^2 \reordtwo \ +\  q^{-1/2} \reordthree
 \end{align}
 \edf
In the above identities, 
the pictures depict the intersection of an $\cN$-tangle with a box 
$ S \times [-1,1]\embed M$, where $S$ is a square and is identified with $S \times \{0\}$. In this box, the $\cN$-tangle is described by its diagram coming from the standard projection onto $S$, which is the shadowed square. The orientation of $S$  is counterclockwise, and the orientation of $M$ is the given by that of $S$ followed by the orientation of $[-1,1]$, which is pointed to the reader. In Equations 
\eqref{eq.arcs0} and \eqref{eq.order0}, the drawn edge of the square with its orientation is an oriented sub-arc of $\cN$.  Besides, the signs indicate the states of each endpoint of the diagram. In all pictures in this paper, the framing is pointing towards the reader except in small neighbourhood of the boundary edges (the oriented arrows) where the framing twists by $\frac{\pi}{4}$ in order to become positively tangent to $\cN$ (up to isotopy there is only one way to achieve this).  Besides, the signs indicate the states of each endpoint of the diagram.

Identity  \eqref{eq.arcs0} with $\epsilon=+, \nu=-$  is an easy consequence of the other relations, see \cite[Lemma 2.3]{Le:TDEC}, but we add it here for a complete list of values of cups (or trivial arcs).

It is clear that an embedding of pairs $i:\MN \to (M' \cN')$ induces an $R$-linear map $i_*: \cS(M,N) \to \cS(M', \cN')$ which depends only on the isotopy class of $i$.

Easy consequences of the defining relations are the following
\begin{align}
\label{eq.kink} -q^{-3} \kinkp \ &= \ \kinkzero\  = \ -q^3\kinkn\\
\label{eq.arcs1} \leftved\  & =\  - q^{3} \delta_{\bar\epsilon, \nu}  C(\epsilon)   \emptyd
\end{align}

\begin{remark} (1) The convention of diagrams near arrowed edges is different from that in
 \cite{Le:TDEC,CL,LY1}: there the marking is  perpendicular to the page and the framing is vertical everywhere, while here the marking (the arrowed interval) is lying flat on the page. There the arrow indicates the height order, but here the arrow is the orientation of $\cN$. However
  the two presentations are canonically equivalent. Our current presentation  is more suitable for the generalisation to marked three manifolds of the present paper. 
 
 (2) If $\cN$ does not have a circle component then our definition of state skein modules coincide with that in \cite{BL,LY1}.
\end{remark}

\def\inv{{\mathsf{inv}}}
\subsection{Orientation inversion of  components of $\cN$} Recall $C(+)=-q^{-\frac{5}{2}}, C(-)=q^{-\frac{1}{2}}$. 
\begin{proposition}\label{prop:inversion} Let $e$ be a connected component of the marking set $\cN$ of a marked 3-manifold $\MN$. Let $\inv_e(\cN)$ be identical to $\cN$ except that the orientation of $e$ is reversed. 
There is an isomorphism of $\cR$-modules $\inv_e:\cSs(M,\cN)\to \cSs(M,\inv_e(\cN))$ defined on a stated $\cN$-tangle $\alpha$ by:
\beq 
\inv_e(\alpha)=\left(\prod_{u \in \alpha\cap e} C(u)\right) \overline{\alpha} \label{eq.inve}
\eeq
where $\overline{\alpha}$ is obtained from $\alpha$ by switching all the states of $\alpha\cap e$ and changing locally near $e$ the framing of $\alpha$ by adding a positive half twist to each component touching $e$. 
Furthermore applying twice $\inv_e$ gives the identity map $\cSs(M,\cN)\to \cSs(M,\cN)$. 
\end{proposition}
\begin{proof}
To show that $\inv_e$ is well defined, we check that Relations \eqref{eq.arcs0} and \eqref{eq.order0} are preserved.

Relation \eqref{eq.arcs0} is preserved, because from the definition and \eqref{eq.arcs1},
$$  \inv_e\left( \leftve\ \right) = C(\epsilon) C(\nu)  \leftvebar =  \delta_{\epsilon, \bar \nu} C(\epsilon)   \inv_e\left( \emptys \right). $$

Consider Relation \eqref{eq.order0}. Apply $\inv_e$ to the left hand side of \eqref{eq.order0},
\beq 
\label{eq.inv2}
\inv_e\left( \reordone\right)=-q^{-3}\reordonepn=
 -q^{-2} \reordtwopn,
 \eeq
 where the last identity follows from \cite[Equ. (20)]{Le:TDEC}. Apply $\inv_e$ to the left hand side of \eqref{eq.order0}, 
 \begin{align}
 \inv_e\left( q^2 \reordtwo \ +\  q^{-1/2} \reordthree\right) & =-q^{-1}\reordnp+\  q^{-1/2} \reordthree \notag \\
 &= 
  -q^{-1}\left(q^{-3}\reordpn+q^{-\frac{3}{2}}(q^2-q^{-2})\reordthree\right)+\  q^{-1/2} \reordthree  \notag \\
  & =-q^{-4}\reordpn+q^{-\frac{9}{2}}\reordthree,
\label{eq.inv3}
 \end{align}
 where the second equality follows from  \cite[Equ. (21)]{Le:TDEC} .
 Comparing \eqref{eq.inv2} and \eqref{eq.inv3}, we see that
Relation \eqref{eq.order0} is transformed into itself.

To prove the last statement observe that the total effect of applying twice $\inv_e$ is to multiply a $\cN$-tangle by $(-q^{-3})^{\# e\cap \alpha}$ and to add a full positive twist to each strand of $\alpha$ near $e$. But each additional framing is equivalent to multiplying $\alpha$ by $-q^{3}$, see \eqref{eq.kink}, so that the different factors compensate. 
\end{proof}
\brem We only need $C(+) C(-)= - q^{-3}$ in the proof.
\erem

\def\id{\mathrm{id}}
\subsection{Manifolds defined up to strict isomorphisms}

We will consider certain geometric operations on 3-manifolds, like cutting and glueing them along disks, or smoothing corners, which produce new manifolds defined up a diffeomorphisms only. Following \cite{LS} we use  the following notion: A {\em strict isomorphism class} of marked 3-manifolds is a family of marked 3-manifolds $(M_i, \cN_i), i \in I$, equipped with diffeomorphisms $f_{ij} : (M_i,\cN_i)\to (M_j,\cN_j)$ for any two indices $i,j$ such that$f_{ii} = \id$ and $f_{jk} \circ f_{ij} = f_{ik}$ up to isotopy. For a strict isomorphism class of marked 3-manifolds we can identify all $\cR$-modules $\cS(M_i, \cN_i)$ via the isomorphisms $(f_{ij})_*$. Note that gluing and cutting operations, or smoothing corner operations produce strict isomorphism classes of marked 3-manifolds.

\def\tal{\tilde \al}
\def\pSi{{ \partial \Sigma}}
\def\ori{{ \mathsf{or}}}

\subsection{Boundary-oriented surface} \label{sec.BOsurface}

\begin{definition}[boundary-oriented surface]\label{def:borderedsurface}

\benu

\item    A boundary-oriented surface is a pair $(\fS, \ori)$ where 

\begin{itemize}
\item $\fS$ is an oriented surface  of finite type, i.e. of the form $\fS =\bfS \setminus P$, where $\bfS$ is a compact surface with possibly empty boundary and $P$ is a finite set, each element of which is called an ideal point of $\fS$,

\item $\ori$ is an orientation of the boundary $\pfS$.

\end{itemize}
 A connected component of $\pfS$ is positive or negative according as the orientation $\ori$ on it is the one induced from the orientation of $\fS$ or not.
 
A non-compact component of the boundary $\pfS$ is called a boundary edge. 

\item A $\pfS$-arc is a proper embedding $[0,1] \embed \fS$.

\item An ideal of $\fS$ is an embedding $(0,1) \embed \fS$ which can be extended to immersion $[0,1] \to \bfS$.

\item a half-ideal arc in $\fS$ is a proper embedding $\al:(0,1] \embed \bfS$ which can be extended to an embedding $\tal: [0,1] \embed \bfS$. Thus $\tal(0)$ is an ideal point while $\al(1)$ is an interior point of a boundary edge.
 
\item An embedding of  boundary-oriented surfaces is a proper orientation preserving embedding which preserves also the orientation of the boundary.

\item The thickening of a boundary-oriented surface $\fS$ is the marked 3-manifold $\MN$, where $M= \fS \times (-1,1)$ and $\cN = \pfS\subset \fS \equiv \fS \times \{0\}$.  Define $\cS(\fS,\ori) = \cS\MN$.

\eenu
\end{definition}

When it is clear from context we write $\fS$ instead of $(\fS,\ori)$. The orientation inversion map $\inv_e$ given by Proposition \ref{prop:inversion} shows that as $R$-modules $\cS(\fS, \ori)\cong \cS(\fS, \ori_+)$, where $\ori_+$ is the positive orientation of $\pfS$.

The projection  $ \fS \times (-1,1)\to \fS$ allows to consider diagrams of $\cN$-tangles.

\bdf
 A $\pfS$-tangle diagram $D$ is a generic immersion of a compact non-oriented 1-dimensional manifold  into $\fS$ in which every double point  is endowed with the under/overcrossing information of the two involved strands. Isotopies of $\pfS$-tangle diagrams  are ambient isotopies of $\fS$. 
\edf

Note that ``generic immersion'' implies $D$ meets $\pfS$ transversally and $D$ has only a finite number of singularity, each is a double point lying in the interior of $\fS$. The empty set is considered as a $\pfS$-tangle diagram.

A $\pfS$-tangle diagram $D$ defines a $\cN$-isotopy class of $\cN$-tangle: equip $D$ with the vertical framing everywhere, except near $\pfS$ one turns the framing by $\pi/4$ to match the orientation $\ori$ of $\pfS$. A  $\pfS$-tangle diagram is 
{\em stated} if it is equipped with a state, which is a map $\partial D \to \{ \pm \}$. A state of $D$ is {\em increasing} if for each boundary edge $c$ of $\partial \fS$, the states on $e$  are increasing when traveling in the direction of $\ori$.

A stated $\pfS$-tangle diagram defines an element of $\cS(\fS,\ori)$. Every $\cN$-isotopy class of  stated $\cN$-tangles can be represented by stated $\pfS$-tangle diagrams. Note that if $D$ is a stated $\pfS$-tangle diagram representing an element $x\in \cS(\fS,\ori)$ and $e$ is a boundary edge of $\fS$, then $D$, up to a power of $q^{1/2}$,  also represents the element $\inv_e(x)\in \cS(\fS,\inv_e(\ori))$, where $\inv_e(\ori)$ is the same as $\ori$ except the orientation of $e$ is reversed.

When $\pfS$ does not have a circle component, we call $\fS$ {\em circle free}. In that case
 $\cS(\fS,\ori)$ has an algebra structure defined in Section \ref{sec:comodule}.

If $\fS$ is circle-free and $\ori=\ori_+$ then $\fS$ is known as a punctured bordered surface in \cite{Le:TDEC,CL,LY1,LY2} and $\cS(\fS,\ori)$  is studied intensively there.

\subsection{Half-ideal slit of surface} \label{sec.slit1}
A $\pfS$-tangle diagram is {\em simple} if  it has neither double points nor trivial components. Here a component is trivial if it is a circle bounding a disk in $\fS$ or it is an arc homotopic relative its endpoints to a subset of $\pfS$. 
 By  \cite[Theorem 2.8]{Le:TDEC} if $\fS$ is circle-free then  $\cS(\fS,\ori)$ is free over $R$ with basis the set  $\cB(\fS,\ori)$ of all isotopy classes of increasingly stated simple diagrams. We want to consider the case when $\pfS$ has a circle component. It turns out that when $\fS$ is non-compact and connected we can show that $\cS(\fS,\ori)$ is free over $R$ and find a free basis of it by eliminating the circles.

 Assume  $\al$ is a half-ideal arc connecting an ideal point $p$ and a point of a boundary component $c\subset \pfS$. Note that $p$ can be an interior ideal point or a boundary ideal point, and $c$ can be a boundary edge or a boundary circle. The {\em $\al$-slit} of $(\fS,\ori)$ is the boundary oriented surface  $(\fS', \ori')$ where  $\fS':=\fS\setminus \alpha$ and  $\ori'$  is the  restriction of $\ori$, see Figure \ref{fig:slit}. The whole interval $\al$ is an ideal point of $\fS'$. We also call $(\fS', \ori')$ {\em a half-ideal slit of $(\fS,\ori)$ breaking $c$}, when we don't want to mention $\al$.
 In $\fS'$ the remnant of $c$ is never a circle.

 \FIGceps{slit}{A half-ideal slit breaking $c$, with an interior ideal point.}{2cm}
 \no{ \tiny
Locally around each $\alpha_i$ we then have the following: 
\begin{center}
\raisebox{-0.7cm}{\includegraphics[width=2cm]{relcirclemarking.pdf}}\put(-30,15){$\fS$}$\to $\raisebox{-0.7cm}{\includegraphics[width=2cm]{relcirclemarking2.pdf}}\put(-30,15){$\fS'$}
\end{center}
where thick oriented edge represents a circle component of $\pfS$  and the thin edge represents  xxx.
}
 
  \begin{theorem}\label{thm.slit1}   Assume
  $(\fS', \ori')$ is the $\al$-slit of a boundary-oriented surface $(\fS,\ori)$, where $\al$ is a half-ideal arc.
The natural embedding $\iota: (\fS', \ori') \embed (\fS, \ori)$ induces an $R$-linear isomorphism $\iota_*: \cS(\fS', \ori') \cong\cS(\fS, \ori)$.
\end{theorem}

\def\id{{\mathrm{id}}}
\def\slitone{  \raisebox{-17pt}{\incleps{1.5 cm}{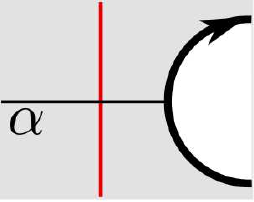}} }
\def\slitonea{  \raisebox{-17pt}{\incleps{1.5 cm}{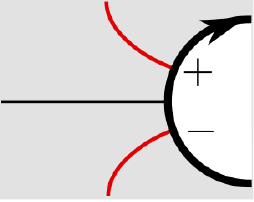}} }
\def\slitoneb{  \raisebox{-17pt}{\incleps{1.5 cm}{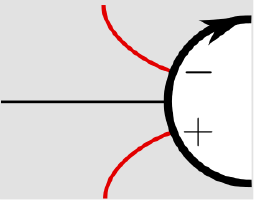}} }
\def\slitonec{  \raisebox{-17pt}{\incleps{1.5 cm}{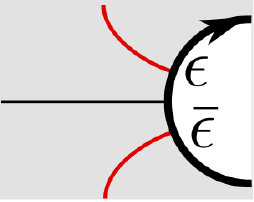}} }
\def\slittwo{  \raisebox{-17pt}{\incleps{1.5 cm}{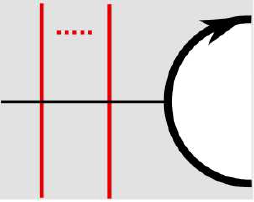}} }
\def\slittwoa{  \raisebox{-17pt}{\incleps{1.5 cm}{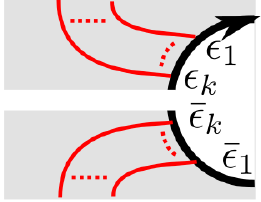}} }
\def\slitthreea{  \raisebox{-17pt}{\incleps{1.5 cm}{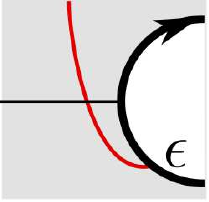}} }
\def\slitthreeaa{  \raisebox{-17pt}{\incleps{1.5 cm}{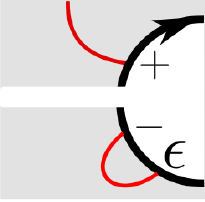}} }
\def\slitthreeab{  \raisebox{-17pt}{\incleps{1.5 cm}{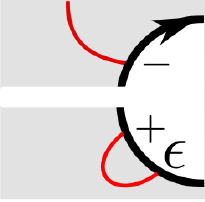}} }
\def\slitthreeac{  \raisebox{-17pt}{\incleps{1.5 cm}{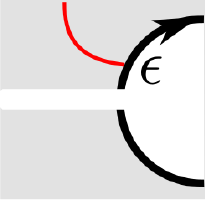}} }
\def\slitthreead{  \raisebox{-17pt}{\incleps{1.5 cm}{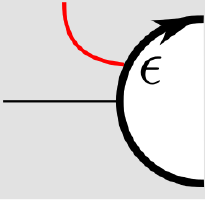}} }
\def\slitthreeb{  \raisebox{-17pt}{\incleps{1.5 cm}{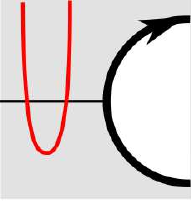}} }
\def\slitthreeba{  \raisebox{-17pt}{\incleps{1.5 cm}{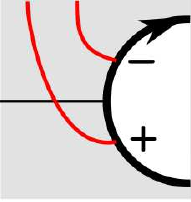}} }
\def\slitthreebb{  \raisebox{-17pt}{\incleps{1.5 cm}{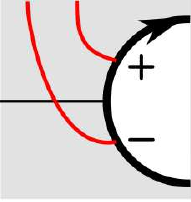}} }
\def\slitthreebc{  \raisebox{-17pt}{\incleps{1.5 cm}{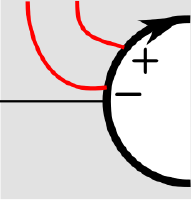}} }
\def\slitthreebd{  \raisebox{-17pt}{\incleps{1.5 cm}{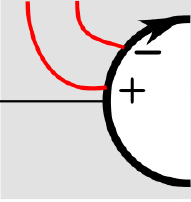}} }
\def\slitthreebe{  \raisebox{-17pt}{\incleps{1.5 cm}{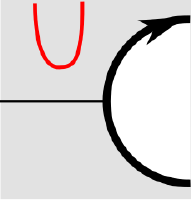}} }

\begin{proof} Using the isomorphism $\inv_e$ we can assume that $\ori$ is the positive orientation.

 Relation~\eqref{eq.order0} can be rewritten as
\beq
\slitone = q^{1/2} \slitoneb - q^{5/2} \slitonea = \sum_{\epsilon = \pm} C(\epsilon)^{-1}\slitonec 
\label{eq.cup1}
\eeq 
which shows that any 
 $x\in\cSs(\fS, \ori)$  is a linear combination of stated $\pfS$-tangle diagrams not meeting $\al$. Hence the map $\iota_*$ is surjective. We will construct an inverse of $\iota_*$.

{\em Claim.} $\cS(\fS,\ori)$ is the free span of isotopy classes of stated $\pfS$-tangle diagrams modulo the four relations \eqref{eq.skein0}--\eqref{eq.order0}. In fact, isotopy classes of stated $\pfS$-tangles are given by isotopy classed of stated $\pfS$-tangle diagrams modulo the Reidemeister moves of type II and type III defined in \cite{Lickorish}. Thus $\cS(\fS,\ori)$ is the
free span of isotopy classes of stated $\pfS$-tangle diagrams modulo the four relations \eqref{eq.skein0}--\eqref{eq.order0} and the Reidemeister moves of type II and type III.
By \cite[Lemma 3.3]{Lickorish} a Reidemeister move of type II or III can be realised by Relations \eqref{eq.skein0} and \eqref{eq.loop0}. Hence we have the claim.

For a concrete stated $\pfS$-tangle diagram $D$  intersecting $\al$ transversally in $k$ points define $f(D) \in \cS(\fS', \ori')$ by repeatedly applying identity \eqref{eq.cup1}:
\beq
f\left ( \slittwo \right): = \sum_{\epsilon_i \in \{ \pm \}} \prod C(\epsilon_i)^{-1}  \slittwoa.
\label{eq.cup2}
\eeq
Here ``concrete'' simply means to we don't identify $D$ with its isotopy class in $\fS$. Let us show that $f$  depends only on the isotopy class of $D$. It is enough to show that $f$ is invariant under the moves M1 and M2 given in Figure \ref{fig:slit3}.

\FIGceps{slit3}{Moves M1 and M2 for isotopy of $D$}{1.5cm}

 Consider move M1. Using  \eqref{eq.cup2} and the values of the cups given by \eqref{eq.arcs0},  
 $$ f\left ( \slitthreea \right) =  q^{\frac12} \slitthreeab - q^{\frac 52} \slitthreeaa  =  \slitthreeac= f\left ( \slitthreead \right).  $$

Consider move M2. Using \eqref{eq.cup2} then move M1, and then \eqref{eq.cup1}, we have
\begin{align*}
 f\left ( \slitthreeb \right) & = q^{\frac12} f\left ( \slitthreeba \right)- q^{\frac52} f\left ( \slitthreebb \right) \\ & = q^{\frac12} f\left ( \slitthreebc \right)- q^{\frac52} f\left ( \slitthreebd \right)
 = f\left ( \slitthreebe \right).
\end{align*} 
More in general if in $D$ there are some vertical strands between the cup shaped strand and $c$ then we first apply \eqref{eq.cup2} to these strands to reduce to the previous case pictured above. 

Thus $f$ is a well-define $R$-linear map. From the  definition $f \circ \iota_*=\id$. It follows that $\iota_*$ is injective, whence bijective.
\end{proof}

When $\fS$ is non-compact and connected, for a circle boundary component $c$ there is a half-ideal arc $\al$  with endpoint in $c$, and the $\al$-slit of $\fS$ is still connected. Hence we have
 \bcor[Basis for the stated skein module of a non-compact surface]\label{teo:basis} Assume a connected non-compact boundary-oriented surface $(\fS,\ori)$ has $k$ circle boundary components. 
After $k$ half-ideal slits breaking  all circle components of $\pfS$ we get a circle-free boundary-oriented surface $(\fS', \ori')$. The embedding $\iota:(\fS', \ori') \embed (\fS, \ori)$ induces an $R$-linear isomorphism $\iota_*:\cS(\fS', \ori') \cong\cS(\fS, \ori)$. In particular $\cS(\fS,\ori)$ is a free $R$-module with basis  $\iota_*(\cB(\fS', \ori'))$. 
\ecor

\brem (1) If the endpoint of the half-ideal arc $\al$ is in a boundary edge $e$, then Theorem~\ref{thm.slit1} is not quite new: it is a reformulation  of a fact proved in \cite[Theorem 4.17]{CL} stating that gluing over a triangle induces isomorphism of stated skein modules. The proof presented here is new even for this special case.

(2) Note that in general the slit isomorphism is not an algebra homomorphism, in case when $\cS(\fS,\ori)$ has the algebra structure, i.e. when $\fS$ is circle-free.
\erem

\subsection{Compact slit}
Corollary \ref{teo:basis} provides a free basis of the $\cR$-module $\cSs(\fS)$ under the hypothesis that $\fS$ is non compact and connected. We will show that when $\fS$ is compact $\cSs(\fS,\ori)$ is a nice  quotient of $\cSs(\fS',\ori')$, where $\fS'$ is non-compact. Besides, Example \ref{ex:disc} will show that in general $\cSs(\fS,\ori)$ is not free over $\cR$, unlike the case when $\fS$ is non-compact and connected.
\bthm \label{thm.slit2} Suppose $(\fS,\ori)$ is a  boundary-oriented surface and $\al$ is a $\pfS$-arc. Let  $\fS'=\fS\setminus \alpha$ and $\ori'$ be the restriction of $\ori$, see Figure \ref{fig:slit4}.
Then $\cSs(\fS,\ori)=\cSs(\fS',\ori')/\sim $ where $\sim$ is the equivalence relation given in Figure \ref{fig:slit4}.  
\ethm

\FIGceps{slit4}{Left: slitting along a properly embedding arc. Right: the equivalence relation. 
All  circular bold arcs  might be in the same  component of~$\pfS$.}{1.5cm}

\def\slitfivea{  \raisebox{-17pt}{\incleps{1.5 cm}{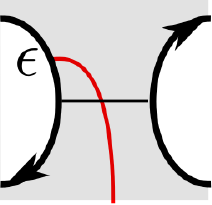}} }
\def\slitfiveb{  \raisebox{-17pt}{\incleps{1.5 cm}{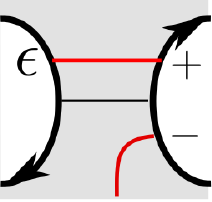}} }
\def\slitfivec{  \raisebox{-17pt}{\incleps{1.5 cm}{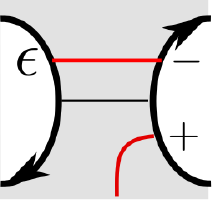}} }
\def\slitfived{  \raisebox{-17pt}{\incleps{1.5 cm}{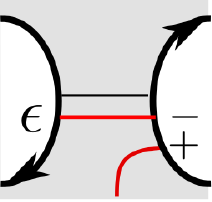}} }
\def\slitfivee{  \raisebox{-17pt}{\incleps{1.5 cm}{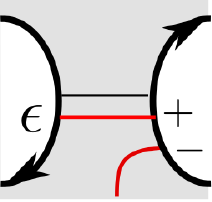}} }
\def\slitfivef{  \raisebox{-17pt}{\incleps{1.5 cm}{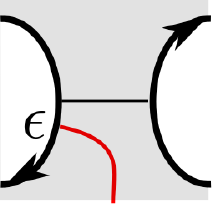}} }

\begin{proof} The proof is similar and uses many ingredients of the proof of Theorem \ref{thm.slit1}.
Again using $\inv_e$ we can assume that the orientation $\ori$ is positive.

The map induced from the embedding $ (\fS', \ori') \embed (\fS, \ori)$ clearly descends to an $R$-linear map $\pi:\cSs(\fS',\ori')/\sim\to \cSs(\fS,\ori)$. Identity \eqref{eq.cup1} shows that $\pi$ is surjective. We will define an inverse of it. Orient $\al$, for example, assuming its direction is pointing to the right  in Figure \ref{fig:slit4}. 

Let $D$ be a concrete stated  $\pfS$-tangle diagram. Define $f(D)$ by exactly the same formula \eqref{eq.cup2}, except now the values should be in $\cSs(\fS',\ori')/\sim$. Note that in defining $f$ we use the right circular arc (determined by the direction of $\al$), not the left one. 

Two stated  $\pfS$-tangle diagrams give isotopic stated  $\pfS$-tangles if and only if they are related by  moves M1, M2, and in addition move M3:
$$ \incleps{1.5cm}{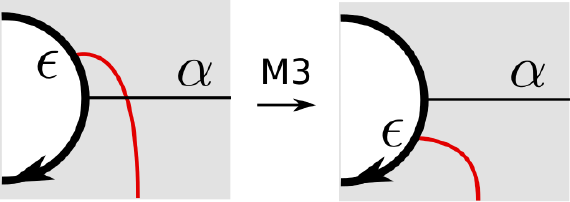}$$
From the proof of Theorem \ref{thm.slit1}, we know $f$ is invariant under moves M1 and M2, even without the relation $\sim$. For move M3 we will need relation $\sim$. Using the definition of $f$ then relation $\sim$,
\begin{align*}
 f\left ( \slitfivea \right)&= q^{\frac12} f\left ( \slitfivec \right)- q^{\frac52} f\left ( \slitfiveb \right)\\
 &= q^{\frac12} f\left ( \slitfived \right)- q^{\frac52} f\left ( \slitfivee \right)= f\left ( \slitfivef \right) 
\end{align*}
Thus $f$ gives a well-defined map  $\cS(\fS, \ori) \to \cSs(\fS',\ori')/\sim$ which is a left inverse of $\pi$. It follows that $\pi$ is injective, whence bijective.
\end{proof}

\subsection{Examples, torsion in  case of compact surfaces} 

\def\PP{{\mathbb P}}
An $n$-gon $\PP_n$ is the standard closed disk with $n$ points on its boundary removed.
 
 Let $\PP_{n,k}$ be  obtained from $\PP_n$ be removing $k$ interior points. 
In particular $\PP_{0,k}$ is the closed disk with $k$ interior points removed. In this subsection we consider $\PP_{n,k}$ as a boundary-oriented surface, where the orientation of the boundary is positive.

In \cite{Le:TDEC} it is proved that $\cS(\PP_{1})\cong R$ via the map whose inverse is $r \to r \emptyset$. \\
We proved in \cite{CL}  that $\cS(\PP_2)$ has a natural structure of a Hopf algebra, and as  Hopf algebras it is isomorphic  to the quantised coordinate algebra $\OSL$ of the group $SL_2(R)$. See also~\cite{KQ}.

By Theorem \ref{teo:basis}, using a half-ideal slit on $\PP_{n,k+1}$
  we get an $R$-linear isomorphism 
 \beq \cS(\PP_{n,k+1} ) \cong \cS(\PP_{n+1,k}).
 \eeq
 In particular we have the following $\cR$-linear isomorphisms:
 $$ \cS(\PP_{0,1})\cong \cS(\PP_{1,0} ) \cong R, \quad \cS(\PP_{1,1})\cong \cS(\PP_{2}) \cong \OSL.$$

If $S_{2,1}$ is the result of removing two small open disks and one point from the sphere, with positive orientation, then by using two half-ideal slits we get $\cR$-linear isomorphism
\beq  \cSs(S_{2,1})=\cSs(\PP_2)=\OSL.
\eeq

\begin{example}[Non-trivial torsion] \label{ex:disc}
Let $\fS=D^2$ be the closed disk with positive boundary orientation. Then by applying Theorem \ref{thm.slit2}, slitting along a diameter, we get that $\cSs(\fS)$ isomorphic to $\cSs(\PP_1)\otimes \cSs(\PP_1)/\sim$ where $\sim$ is defined in Figure \ref{fig:slit4}.
In $\PP_1$ the top red arc in the figure is equivalent to $0$ if its states are equal, to $q^{-1/2}$ if the left state is $+$ and the right one $-$ and finally to $-q^{-5/2}$ in the remaining case. 
The bottom red arc is instead equivalent to $0$ if its states are equal, to $-q^{-5/2}$ if the left state is $+$ and the right one $-$ and finally to $q^{-1/2}$ in the remaining case. 
Therefore we get $\cSs(\fS)=\cR/(q^{-1/2}+q^{-5/2})=\cR/(1+q^{2})$.

\end{example}

\subsection{Circle boundary element}

\begin{lemma}\label{lem:circularskeins} Suppose $c$ is a circle boundary component of an oriented surface $\fS$. Assume a stated $\pfS$-tangle diagram $\al$ is the disjoint union $\al= \al_1\sqcup \al_2$ where $\al_1$  is a simple closed curve parallel to $c$. Then 
as elements of $\cS(\fS)$ we have $\alpha=2\al_2\in \cSs(\fS)$. 
\end{lemma}

\def\AA{{\mathbb A}}
\def\slitten{  \raisebox{-14pt}{\incleps{1.3 cm}{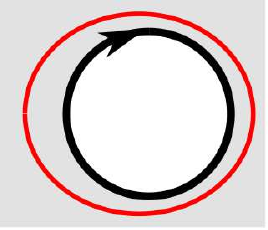}} }
\def\slittena{  \raisebox{-14pt}{\incleps{1.3 cm}{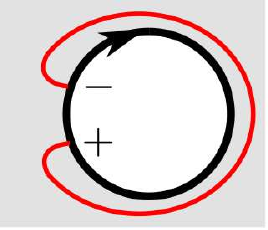}} }
\def\slittenb{  \raisebox{-14pt}{\incleps{1.3 cm}{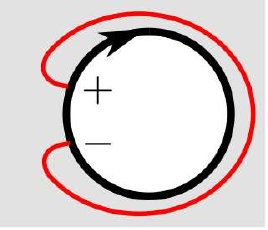}} }
\def\slittenc{  \raisebox{-14pt}{\incleps{1.3 cm}{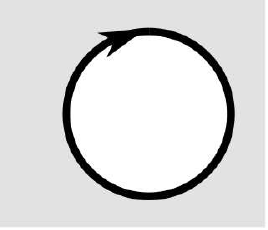}} }
\begin{proof} Using \eqref{eq.cup1}, we have
$$ \slitten = q^{\frac 12}\slittena -  q^{\frac 52}\slittenb = 2\,  \slittenc $$
 where in the second equality we used the values of cups given by \eqref{eq.arcs0}.
\end{proof}

\def\ofS{\mathring {\fS}}

  \def\dec{\mathrm{dec}}
  
  \def\ufS{\underline{\fS}} 
\subsection{Cutting homomorphism} Two boundary components of a boundary-oriented surface  are {\em of the same type} if they are both circles or both boundary edges.
Let  $c_1, c_2$ be two boundary components of the same type of a boundary-oriented surface $(\fS',\ori')$.  Assume  $c_1$ and $c_2$ have opposite orientations, i.e. one positive and one negative. Let $\fS= \fS'/(c_1=c_2)$ where we identify $c_1$ with $c_2$ via an orientation preserving diffeomorphism. Let $\ori$ be the orientation of $\pfS$ which is induced from $\ori'$, and $\pr: \fS' \to \fS$ be the natural projection. 
Denote $c= \pr(c_1)= \pr(c_2)$. If $c_1, c_2$ are boundary edges then $c$ is an oriented ideal arc of $\fS$, otherwise  $c$ is a an oriented simple closed curve in the interior of $\fS$.

In this situation we also say that $(\fS',\ori')$ is a result of cutting $(\fS,\ori)$ along $c$. 
 For an example see Figure \ref{fig:cut}, where we also give an idea of how the map $\Cut_c$ is defined.  

\FIGceps{cut}{Left: Cutting along an interior ideal arc $c$. Right: The map $\Cut_c$. The case when $c$ is a circle is similar.}{2.3cm}

Suppose $\al$ is a stated $\pfS$-tangle diagram transversal to $c$. Then  $\tal:=\pr^{-1}(\al)$ is a $\pfS'$-tangle diagram which inherits states from $\al$ at all boundary points, except for those in $c_1 \cup c_2$. For every $\bove: \al \cap c \to \{\pm\}$ let $\tal(\bove)$ be the stated $\pfS'$-tangle diagram whose states on $c_1 \cup c_2$ are the lift of $\bove$, i.e. the state at both points in $\pr^{-1}(u)$ is $\bove(u)$.

Here is an extension of \cite[Theorem 1]{Le:TDEC}, where the case $c$ is an ideal arc is proved.
\begin{theorem}\label{teo:cutting} Assume $(\fS',\ori')$ is a result of cutting $(\fS,\ori)$ along $c$ as above, where $c$ is either an interior oriented ideal arc or an interior oriented simple loop.

 There exists a unique $\cR$-linear homomorphism $\Cut_c :\cSs(\fS,\ori) \to \cSs(\fS',\ori')$ such that if $\al$ is a stated 
 $\pfS$-tangle diagram transversal to $c$ then
 \beq 
  \Cut_c(\al) = \sum_{\bove: \al \cap c \to \{ \pm \}} \tal(\bove) \in \cS(\fS').
  \label{eq.cut0}
  \eeq

\end{theorem}

\begin{proof} The proof for the case when $c$ is an ideal arc is given in \cite{Le:TDEC} and can be applied also to the case when $c$ is a circle:
for a concrete stated $\pfS$-tangle diagram $D$  define $\Cut_c \in \cS(\fS', \ori')$ by the right hand side of \eqref{eq.cut0}. We need to show that $\Cut_c(D)$ depends only on the isotopy class of $D$. It is enough show that $\Cut_c(D)$ is invariant under the move
$$ \incleps{1.3cm}{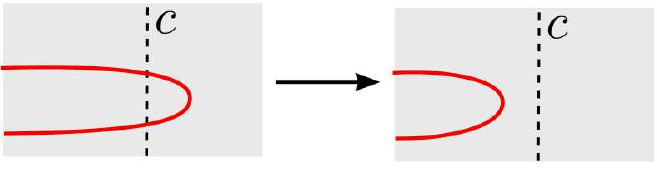}.$$
This is proved in \cite{Le:TDEC} for the case when $c$ is an ideal arc, but the proof there involves only a small part of $c$ and applies as well in the case when $c$ is a circle.
\end{proof}
 From the definition we see that if $c$ and $c'$ are disjoint then
 \beq  \Cut_{c} \circ \Cut_{c'} = \Cut_{c'} \circ \Cut_{c}. \label{eq.commu5}
 \eeq
\brem In \cite{Le:TDEC} it is proved that if $\fS$ is circle free then $\Cut_c$ is injective. However, cutting along a circle might not be injective. In fact, let $\fS$ be an arbitrary circle-free boundary-oriented surface and $c$ be a trivial simple loop in $\fS$. Cutting $\fS$ along $c$ we get  $\fS'=\fS_1 \sqcup \fS_2$, where $\fS_2$ is a closed disk. By Example \ref{ex:disc}, as $R$-modules $\cS(\fS_2)= R/(q^2+1)$. The cutting homomorphism
$$ \Cut_c: \cS(\fS) \to \cS(\fS') = \cS(\fS_1) \ot_R (R/(q^2+1))$$
is not injective since it maps the empty element $\emptyset$, which is an element of the free $R$-basis $\cB(\fS)$ of $\cS(\fS)$, to a torsion element killed by $q^2+1$. 
\erem

\subsection{Cutting for $3$-manifolds}  \label{sub:gluing} Cutting for 3-manifolds is similar. The case involving a boundary edge is discussed in \cite{BL,LY1}. Let us consider the general case.

Suppose $(M', \cN')$ is a marked 3-manifold, not necessarily    
connected. Assume $c_1, c_2\subset \partial M'$ are two distinct components of $\cN'$ of the same type (i.e. both arcs or circles) and let
$D_1,D_2\subset \partial M'$ be closed disjoint regular neighbourhoods of $c_1$ and $c_2$. This means, if  $c_1, c_2$ are boundary edges then each $D_i$ is a closed disk containing $c_i$ in its interior, otherwise each $D_i$ is a closed annulus containing $c_i$ in its interior and deformation retracts to $c_i$. We assume that $D_i\cap \cN= c_i$.
Choose an  orientation-reversing diffeomophism $\phi:D_1 \to D_2$ such that $\phi(c_1)=c_2$ as oriented arcs or circles. 
Let $M$ be obtained from $M'$ by gluing $D_1$ to $D_2$ via $\phi$ and $\pr:M' \to M$ be the canonical projection. Denote $c= \pr(c_1) =\pr( c_2)$ and $D= \pr(D_1)= \pr(D_2)$. Orient $c$ using the orientation of $c_1$ (or $c_2$).
Consider the marked 3-manifold $\MN$ where
$\cN = \pr^{-1}( \cN' \setminus ( c_1 \cup c_2))$.

An $\cN$-tangle $\al$ in $M$ is {\em $(D,c)$-transversal} if 
\begin{itemize}
\item $\al$ is transversal to $D$,
\item $\al \cap D=\al \cap c$, and
\item the framing at every point of $\al\cap c$ is a positive tangent vector of $c$.
\end{itemize}

It is easy to see that every $\cN$-tangle is $\cN$-isotopic to one which is $(D,c)$-transversal. 

Suppose $\al$ is a  $(D,c)$-transversal stated $\cN$-tangle. Then $\tal:=\pr^{-1}(\al)$ is an $\cN'$-tangle which is stated at every boundary point except for the boundary points in  $c_1\cup c_2$. For every map $\bove: \al \cap c \to \{\pm\}$ let $\tal(\bove)$ be the stated $\cN'$-tangle where the state of a boundary point $u\in N \cup N'$ is $ \bove(\pr(u))$. 

\begin{theorem}\label{thm.gluing} With the above assumptions,
there is a unique $R$-linear homomorphism
$\Cut_{D,c}: \Ss(M,\cN) \to \Ss(M', \cN')$ such that for every $(D,c)$-transversal stated $\cN$-tangle $\al$,
$$
\Cut_{D,c}(\al)= \sum_{\bove: \al \cap c \to \{\pm \}}  \tal(\bove), \quad \text{identity in } \ \Ss(M', \cN').
$$

Furthermore if $(D',c')$ is another pair as above, so that $D\cap D'=\emptyset$ then $\Cut_{D,c}\circ \Cut_{D',c'}=\Cut_{D',c'}\circ \Cut_{D,c}$.

\end{theorem}
\begin{proof}
One needs to prove that the map is well defined. This is a local statement where this verification is identical to that performed in \cite{Le:TDEC}. The proof of all the statements is identical to that given in  \cite{Le:TDEC}. 
\end{proof}

\no{

Moreover the map $\varphi$ is associative in the following sense. Suppose $N_1,N_1',N_2, N'_2\subset \partial M$, are four distinct components of $\cN$ and let $$M_i=M/(N_i=N_i'), \cN_i=pr_i(\cN\setminus \{N_i,N'_i\}), i=1,2\ {\rm and}\ M_{12}=M_1/(N_2=N'_2)=M_2/(N_1=N'_2).$$ 

Let $$\varphi_i:\Ss(M_i,\cN_i)\to \Ss(M,\cN), \varphi_{12}^i:\Ss(M_{12},\cN_{12})\to \Ss(M_i,\cN_i),\  i=1,2$$
be the $\cR$-morphisms described above. 
Then it holds $\varphi_1\circ \varphi^1_{12}=\varphi_2\circ \varphi^2_{12}$.
}

\begin{lemma}
Let $(M,\cN)$ be a marked $3$-manifold and suppose that $c\in \cN$ is a circle component. Let $\alpha\subset M$ be the framed link isotopic to $c$ with the framing tangent to $\partial M$ along $c$. 
Then $[\alpha]=2[\emptyset] \in \cSs(M,\cN)$. 
\end{lemma}
\begin{proof} 
A tubular neighbourhood of $c$ in $M$ is homeomorphic to the thickening of an annulus with one circular marked boundary. 
Then the statement follows from Lemma \ref{lem:circularskeins}.
\end{proof}
By Example \ref{ex:disc} we have the following: 
\begin{lemma}
If $(M,\cN)$ is as above and $c\in \cN$ is a circle component bounding a disc in $\partial M$ then $(q^2+1)[\emptyset]=0\in \cSs(M,\cN)$.
\end{lemma}

\section{Non-injectivity of several natural maps}

In this section we show that several  homomorphisms between skein modules, which are injective in surface cases, are not injective in 3-manifold cases.

For a non-zero complex number $q^{1/2}$  denote $\Sqq:=\cS\MN$ where  the ground ring is $\cR=(\BC,q^{1/2})$.
 When $\cN=\emptyset$ we denote $\Sqq$ by $\Sq(M)$. Note that in this case $\Sq(M)$ depends only on $q$.

A complex number $q$ is a {\em root of 1} if there is a positive integer $d$ such that $q^d=1$, and the least such positive integer is denoted by $\ord(q)$. The quantum integer is defined by
$$ [n]_q= \sum_{i=-n+1}^{n-1} q^{2i}= \frac{q^{2n} - q^{-2n}}{q^{2} - q^{-2}}.$$
The smallest positive integer $N$ such that $[N]_q=0$ is equal $\ord(q^4)$ as long as $\ord(q^4)>1$.
 For this reason we often use $\ord(q^4)$ instead of $\ord(q)$.

\def\TL{\mathsf{TL}}

\subsection{Pattern in a disk} \label{sec.JW} Let $D$ be the standard closed disk and $W_n\subset \partial D$ be a set of $2n$ points in its boundary. A {\em $W_n$-tangle diagram} $T$ on $D$ is a generic embedding of a compact non-oriented 1-dimensional manifold into $D$ such that $\partial T= W_n$, with the usual under/overcrossing information at every double point like in a knot diagram.  We consider $T$ as a framed tangle in $\tilde D:=D \times(-1,1)$, with vertical framing everywhere.
Define $\TL_{n}$ as the $\BC$-module generated by   isotopy classes of $W_n$-tangle diagrams modulo the skein relations \eqref{eq.skein0} and \eqref{eq.loop0}. Note that $\TL_n$, known as the Temperley-Lieb algebra, depends on $q$ but we suppress $q$ in the notation.
An element $x\in \TL_n$ is called a {\em pattern}.

Suppose $x= \sum c_i T_i\in \TL_n$, where each $T_i$ is a $W_n$-tangle diagram. An element $\al\in \cS\MN$ is a {\em closure of $x$} if there is an embedding of the thickening  $\tilde D:=D \times(-1,1)$ into $M$ such that $\al$ has a presentation $\al = \sum c_i \al_i$  where each $\al_i$ is a stated $\cN$-tangle and $\al_i\cap \tilde D= T_i$, and outside $ \tilde D$ all the tangles $\al_i$ are the same. If we denote the common outside part by $\beta$, then we say that $\al$ is the result of closing $x$ by $\beta$.

\def\JW{\raisebox{-8pt}{\incleps{.8 cm}{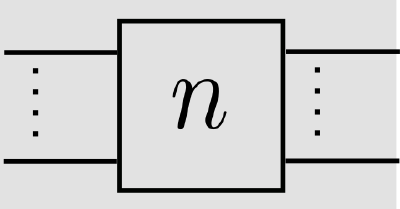}}}
\def\JWd{\raisebox{-8pt}{\incleps{.8 cm}{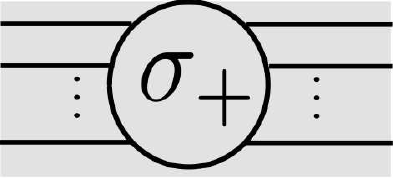}}}
\def\JWzero{\raisebox{-8pt}{\incleps{.8 cm}{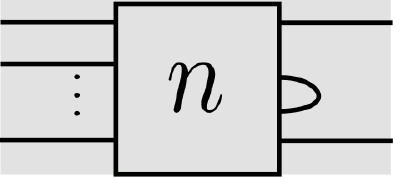}}}
\def\BA{{\mathbb A}}
\def\oBA{\mathring{\mathbb A}}
\def\Sym{{\mathrm{Sym}}}

For each $n \ge 0$ the {\em Jones-Wenzl idempotent} is the element $f_n\in \TL_n$, denoted by a box enclosing $n$, and  defined by
$$ f_n = \JW := \frac1{[n]_q!} \sum _{\sigma \in \Sym_n} q^{\ell(\sigma)} \JWd ,$$
where $\Sym_n$ is the group of permutations of $n$ objects, and $\sigma_+$ is the positive braid with minimal number of crossing representing the permutation $\sigma$, and $\ell(\sigma)$ is the length of $\sigma$. For the definition $f_n$ we must assume that $[n]_q!$ is invertible in $\cR$. 
It is known that, see \cite[Lemma 13.2]{Lickorish}, $f_n$ has the non-returnable property:
\be  \JWzero =0,  \label{eq.zero}
\ee
where the cap connect two consecutive right boundary points of the box.

Let $\oBA=(-1,1)\times S^1$ be the open annulus. The core of $\oBA$ is the circle $a=\{0\} \times S^1$. The skein algebra $\cS(\oBA)$ is equal to the ring $R[a]$ of polynomials in $a$. By \cite[Lemma 13.2]{Lickorish}, in   $\cS(\oBA)$ we have
\beq 
\raisebox{-20pt}{\incleps{1.5cm}{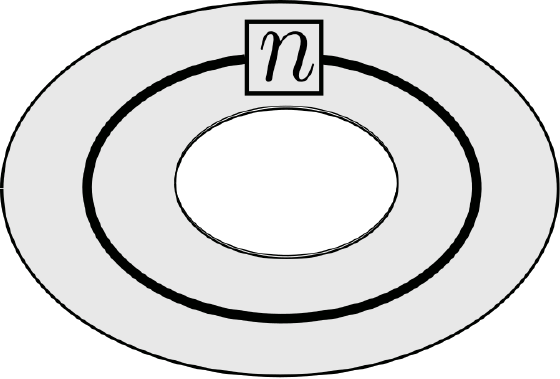}} = S_{n}(a),
\label{eq.JWc}
\eeq
where 
 $S_n(x)\in \BZ[x]$ is the Chebychev polynomial of second type defined inductively by
$$ S_0(x)=1, S_1(x)= x, S_n(x) = x S_{n-1}(x) - S_{n-2}(x) \ \text{for} \ n \ge 2.$$

\def\bM{{\mathbf M}}
\subsection{Connected sum}  \label{sec:connected}
For $i=1,2$ assume $\bM_i=(M_i,\cN_i)$ is a connected marked 3-manifold.
Recall that the connected sum $M_1 \# M_2$  is obtained by first removing the interior of a small ball $B_i$ from $M_i$ to obtain $M'_i$ then gluing $M'_1$ with $M'_2$ along the boundaries of $B_i$. Let $\bM_1 \# \bM_2= (M_1\# M_2, \cN_1 \cup \cN_2)$.
Define
$$ \Psi_{\bM_1, \bM_2;q^{1/2}} : \Sq(\bM_1) \ot \Sq(\bM_2) \to \Sq(\bM_1 \# \bM_2)$$
so that if $\al_i\subset M_i$ is a framed tangle not meeting $B_i$, then 
$$ \Psi_{\bM_1, \bM_2;q^{1/2}}(\al_1 \ot \al_2) = \al_1 \cup \al_2, \ \text{as an element of $\Sq(\bM_1 \# \bM_2)$}.$$

It is easy to see that $\Psi_{\bM_1, \bM_2;q^{1/2}}$ is a well-defined $\BC$-linear homomorphism.

J. Pryztycki \cite{Prz}  proved that  if $q$ is not a root of 1 and $\cN_1=\cN_2=\emptyset$, then  $\Psi_{\bM_1, \bM_2;q^{1/2}}$ is  bijective. The proof can be easily extended to the case of arbitrary $\cN_1$ and $\cN_2$ using Proposition \ref{prop:sphere}.
Here we show that in general the map $\Psi_{\bM_1, \bM_2;q^{1/2}}$ is not injective.

\def\bA{{\mathcal A}}
\def\Fr{\mathrm{Fr}}
\def\BA{\mathbb{A}}
\def\BH{\mathbb{H}}

Assume $q$ is a root of 1. For a marked 3-manifold $\bM=\MN$ let
 $F_{q^{1/2}}(\bM)$ be the $\BC$-subspace of $\Sq(\bM)$ spanned by all closures the Jones-Wenzl idempotent $f_{N-1}$, where  $N = \ord(q^4)$.

\bthm [Proof in Subsection \ref{sec.proof}]  \label{thm.connected3}
 Assume $q$ is a complex root of 1 with $\ord(q^4)=N>1$. Then $F_{q^{1/2}}(\bM_1) \ot F_{q^{1/2}}(\bM_2)$ is in the kernel of $\Psi_{\bM_1, \bM_2;q^{1/2}}$. 
\ethm
\brem The proof actually shows that the statement is true over any ground ring, assuming $\ord(q^4)=N >1$ and $[N-1]_q!$ is invertible, so that $f_{N-1}$ can be defined.
\erem
In particular if $N=2$ we get the following: 
\begin{corollary}\label{cor:connected2}
Suppose $\ord(q^4)=2$. If for $i=1,2$ $\alpha_i\subset \bM_i$ is a non-empty $\cN_i$-tangle then $$\Psi_{\bM_1, \bM_2;q^{1/2}}(\alpha_1\sqcup \alpha_2)=0.$$ 
\end{corollary}
\begin{proof}Since $\ord(q^4)=2$, the subspace $F_{q^{1/2}}(\bM_i)$ is spanned by closures of $f_1$, which is the same as a plain strand. 
\end{proof}

We expect that if  $\ord(q^4) >1$ and $\pi_1(M)$ is non-trivial then $F_{q^{1/2}}(M) \neq 0$. This is true at least for thickened surfaces:

\bpro  \label{r.nonz} Assume   $\fS$ is a circle free boundary-oriented surface with non-trivial fundamental group, and $q\in \BC$ is a root of 1 with $\ord(q^4) >1$.
Then $F_{q^{1/2}}(\fS) \neq \{0\}$. More precisely for any non-trivial simple closed curve $\al\subset \fS$ we have $0\neq S_{N-1}(\al) \in F_{q^{1/2}}(\fS)$. 
\epro

\begin{proof} For $n \in \BN$ the elements $\al^n\in \cS(\fS)$, presented by $n$ parallel copies of $\al$, are distinct elements of the free  basis $\cB(\fS)$ of $\cS(\fS)$ described in Subsection \ref{sec.slit1}. Hence  $S_n(\al)\neq 0$ for all $n$. By \eqref{eq.JWc} the skein $S_n(\al)$ is a closure of $f_{n}$.
 Thus $0\neq S_{N-1}(\al) \in F_{q^{1/2}}(\fS)$. 
 \end{proof}
\bcor For $i=1,2$ suppose $M_i= \fS_i \times (-1,1)$ where $\fS_i$ is a circle free boundary-oriented surface with non-trivial fundamental group. Let $x_i\subset \fS_i$ be a non-trivial simple closed curve. Assume $q$ is a root of 1 with $\ord(q^4)=N >1$. Then $S_{N-1}(x_1) \otimes S_{N-1}(x_2)$ is a non-zero element of the kernel of  $\Psi_{M_1, M_2;q^{1/2}}$.
\ecor

A special useful case is when $M_1, M_2$ are the thickening of the annulus $\oBA= (0,1)\times S^1$.

\bcor \label{r.connected4}
Suppose $M_1\cong M_2 \cong \oBA \times (-1,1)$. Let $x_i$ be the core of $M_i$. Assume $\ord(q^4)=N >1$. Then $S_{N-1}(x_1) \otimes S_{N-1}(x_2)$ is a non-zero element of the kernel of $\Psi_{M_1, M_2;q^{1/2}}$.
\ecor

\begin{remark}\label{rem.connected}
 In an earlier version of the paper we proved Corollary \ref{cor:connected2} which shows the non-injectivity of $\Psi_{M_1,M_2; q^{1/2}}$ for the case $\ord(q^4)=2$, for a large class of 3-manifolds. Then,
 answering the second author's question about a generalisation to higher order roots, H. Karuo \cite{Karuo} proved a weaker version of Corollary \ref{r.connected4}, showing that the kernel of $\Psi_{M_1, M_2;q^{1/2}}$ contains a 
  polynomial in $\BC[x_1, x_2]$ with highest term $x_1^{N-1} x_2^{N-1}$. Here  we have an explicit formula for a polynomial in the kernel, and we will use this explicit formula in the proof of the non-injectivity of the Chebychev-Frobebius homomorphism, see Theorem~\ref{thm.Frobenius}.
\end{remark}

\def\bSi{\boldsymbol{\Sigma}}
\subsection{Empty tangle element} If $\fS$ is a circle free boundary-oriented surface then the empty tangle, being an element of the free basis $\cB(\fS)$, is not zero and moreover serves as the unit of the algebra structure.

The situation can change for 3-manifolds. Suppose $\MN$ is a marked 3-manifold. 
We say that an embedded sphere $S=S^2 \embed M$ lying in the interior of $M$ is {\em marking separating } if there there is a properly embedded path $a: [0,1] \embed M$ transversal to $S$ and meeting $S$ at exactly one point such that $a(0), a(1) \in \cN$.

\bthm [Proof in Subsection \ref{sec.proof}] \label{thm.zero}
Assume  a marked 3-manifold $\MN$ has a marking separating sphere $S$, and $q$ is a complex root of 1 such that $\ord(q^4) >1$. Then any stated $\cN$-tangle not meeting $S$ is equal to 0 in $\Sq\MN$. In particular the empty tangle is zero.
\ethm

\subsection{Non-injectivity of the cutting homomorphism}

  For surfaces the cutting homomorphism along an ideal arc is always injective, see \cite{Le:TDEC}. 
 \begin{theorem}[Proof in Subsection \ref{sec.proof}] \label{thm.CH}
Suppose $q$ is a complex root of 1 with $\ord(q^4) >1$.  
There exists a compact 3-manifold $M$, a properly embedding disk $E \embed M$, and an oriented open interval $e\subset E$ such that the cutting homomorphism 
$$\Cut_{E,e}: \Sq(M) \to \Sq(M', \cN')$$
 is not injective. Here $(M', \cN')$ is the result of cutting $(M, \emptyset)$ along $(E,e')$.
 \end{theorem}

 \def\hM{\hat M}
 \def\hcN{\widehat {\cN}}
 \def\hbM{\hat {\bM}}

 \subsection{Non-injectivity of adding a marking} Let $\bM=\MN$ be a marked 3-manifold where $M$ is connected. Choose a closed ball $B$ in the interior of $M$. Let $\hbM = (\hM, \hcN)$ where $\hM = M \setminus \mathring B$ and $\hcN= \cN \cup c$, where $c$ is an open interval on $\partial B$. Define the $R$-linear map $\Gamma_\bM: \cS(\bM) \to \cS(\hbM)$ as follows. Suppose $\al\in \cS(\bM)$ is represented by a stated $\cN$-tangle $T$. By an isotopy we can assume $T$ does not meet $B$. Then  $\Gamma_\bM(\al)=T$ as an element $\cS(\hbM)$. It is easy to see that $\Gamma$ is well-defined. This construction is closely related to the notion of quantum fundamental group discussed in Subsection \ref{sub:results}.

 \begin{theorem} Assume $q$  is a complex root of unity with $N=\ord(q^4) >1$. There exists a marked 3-manifold $\bM=\MN$  such that $\Gamma_\bM$ is not injective.
 \end{theorem}
 
 \begin{proof} We present here two independent proofs.

(i) Let $M$ be the closed 3-ball and $\cN$ be an open interval on $\partial M$. The $\hM= S^2 \times [1,2]$ and 
 $\hcN$ consisting of two intervals $e_1, e_2$ where $e_i \subset S^2 \times \{i\}$.
 Clearly the sphere $S^2 \times \{3/2\}$ is separating $e_1$ and $e_2$. By Theorem \ref{thm.zero} the empty tangle is equal to 0 in $\Sq(\hbM)$. Since $\bM$ is the thickening of the monogon $\PP_1$,  the empty tangle is not 0 in $\Sq(\bM)= \BC$. 
 
 (ii)
  The following proof gives a much larger class of examples. 
  First assume $\bM$ be any marked 3-manifold.
 We have the following commutative diagram 
 \be 
\begin{tikzcd}
\Sq(\bM ) \ot \Sq(\bM)\arrow[r,"\Psi_{\bM, \bM;q^{1/2}}"]\arrow[d,"\Gamma_{\bM} \ot \Gamma_{\bM} "] & \Sq(\bM \# \bM)\arrow[d,"\Gamma_{\bM \# \bM}   "] \\
\Sq (\hat{\bM})  \ot  \Sq (\hat{\bM}) \arrow[r,"\cong "]  &  \Sq (\widehat{\bM \# \bM})\\
\end{tikzcd}
\label{eq.dia1}
\ee 
where the lower map is the isomorphism of Theorem \ref{teo:fundgroup}. Corollary \ref{r.connected4} showed that there are examples when $\Psi_{\bM, \bM;q^{1/2}}$ is not injective. In that case the commutative diagram implies that 
 $\Gamma_{\bM} \ot \Gamma_{\bM}$ is not injective, which in turns,  implies that  $\Gamma_{\bM}$ is not injective.
 \end{proof}

\subsection{Non-injectivity of the Chebyshev-Frobenius homomorphism}
Suppose $q^{1/2}$ is a root of 1 and $N= \ord (q^4)$. Let $\epsilon := q^{N^2/2}$. Note that $\epsilon^8=1$.

The Chebyshev polynomial of first type $T_n(x)=\sum_{i=0}^N c_i x^i\in \BZ[x]$ is defined  by the identity
$$ T_n(u+u^{-1}) = u^n + u^{-n}.$$
 
For a framed knot $\al$ in an oriented 3-manifold $M$ define the {\em $T_N$-threading} of $\al$ by
$$ \al^{(T_N)} = \sum_{i=0}^N c_i \al^{(i)}, \text{considered as an element of } \ \Sq(M),$$
where $\al^{(i)}$ is $i$ parallel push-offs (using the framing) of $\al$ lying in a small neighbourhood of~$\al$. When $\al$ is the disjoint union of $k$ framed knots, $\al= \sqcup_{i=1}^k \al_i$, its threading is defined by linear extrapolation:
$$ \al^{(T_N)} = \al_1^{(T_N)} \cup \dots \cup \al_k^{(T_N)}:= \sum_{i_1,\dots, i_k=0}^N c_{i_1} \dots c_{i_k} \left [ \al_{1}^{(i_1)} \cup \dots \cup \al_{k}^{(i_k)} \right].$$

The Chebyshev-Frobenius homomorphism is  the $\BC$-linear map 
$$\Phi_{q^{1/2}}: \cS_\epsilon (M) \to \Sq (M)$$ defined so that if $x\in \cS_\epsilon(M)$ is  presented by disjoint union $\al$  of framed knots then
\be 
\Phi_{q^{1/2}}(\al) =  \al^{(T_N)} \quad \text{considered as an element of } \ \Sq(M).
\ee
The well-definedness of $\Phi_{q^{1/2}}$ is not an easy fact. When $M$ is a thickened surface without boundary Bonahon and Wong \cite{BW} showed that $\Phi_{q^{1/2}}$ is well-defined. The result is extended to all 3-manifolds in \cite{Le3}. For the case of marked 3-manifolds see \cite{BL,LP}, where the definition of $\Phi_{q^{1/2}}$ needs to be modified for arcs. 
When $M$ is the thickening of a surface $\fS$ without boundary  $\Phi_{q^{1/2}}$ is injective as it maps the basis $\cB(\fS)$ of $\cS_\epsilon(\fS)$ injectively into a basis of $\Sq(\fS)$. Here we show that $\Phi_{q^{1/2}}$ is not injective in general.

\begin{theorem} \label{thm.Frobenius}

Let $q$ be a complex root of 1 with $\ord(q^4)= N >1$. There exists a compact oriented 3-manifold $M$ such that 
Chebyshev-Frobenius homomomorphism 
$\Phi_{q^{1/2}}: \cS_\epsilon(M) \to \Sq(M)$ is not injective.

\end{theorem}

\begin{proof}  Let $M= M_1  \#M_2$ where each $M_i$ is a 
thickened annulus $\BA \times[-1,1]$, a solid torus. Let $x_i$ be the core of $M_i$. Recall that $\Psi_{M_1,M_2;\epsilon }$ is the connected sum homomorphism (Subsection \ref{sec:connected}).
Define 
$$x= \Psi_{M_1,M_2;\epsilon }((x_1^2-4) \ot (x_2^2-4)) \in \cS_\epsilon(M) .$$
 By definition, 
$$ \Phi_{q^{1/2}}(x)= \Psi_{M_1,M_2;q^{1/2}} ( (T_N(x_1)^2-4) \ot (T_N(x_2)^2-4)) \in \Sq(M) .$$
Let us show that  $T_n(x)^2-4 \in S_{n-1}(x) \BZ[x]$. Embed $\BZ[x] \embed \BZ[u^{\pm 1}]$ by $x=u+ u^{-1}$. 
 Then
 $$ T_n(x)^2-4 = (u^n + u^{-n})^2-4=(u^n - u^{-n})^2 = (u-u^{-1})^2 S_{n-1}(x)^2 \in S_{n-1} (x) \BZ[x].$$
Thus $(T_N(x_1)^2-4) \ot (T_N(x_2)^2-4)\in F_{q^{1/2}}(M_1) \ot F_{q^{1/2}}(M_2)$. By Corollary \ref{r.connected4} we have $\Phi_{q^{1/2}}(x)=0$.

It remains to show $x\neq 0$ in $\cS_\epsilon(M)$.

First we assume $\epsilon^2= \pm 1$. In this case $\cS_\epsilon(M)$ has the structure of a commutative algebra where for two disjoint framed links $\al$ and $\beta$ in $M$ the product $\al\beta$ is the disjoint union $\al \sqcup \beta$. As a $\BC$-algebra $\cS_\epsilon(M)$ is isomorphic to the universal $SL_2$-character variety of $M$, see \cite{Bullock,PS1}. In particular there is a surjective algebra homomorphism $\Omega: \cS_\epsilon(M) \to \BC[\chi(M)]$, where $\chi(M)$ is the $SL_2(\BC)$-character variety of the fundamental group $\pi_1(M)$. The fundamental group of $M$ is free on two generators $z_1$ and $z_2$ where $z_i$ is a loop representing the core of $M_i$. It is known that $\BC[\chi(M)]$ is the ring of polynomials in 3 variables $u_1=\tr(z_1), u_2=\tr(z_2)$ and $u_{12}=\tr(z_1z_2)$. 
In particular, we have an embedding $\BC[u_1, u_2 ] \embed \BC[\chi(F_2)]$. By definition,  $\Omega(x_1)= \sign(x_i) u_i$, where $\sign(x_i) \in \{\pm1\}$ whose exact value is not important as $\Omega(x_1^2)= u_i^2$. 
It follows that 
$$ \Omega(x) = (u_1^2-4)(u_2^2-4) 
\neq 0 \ \text{in } \BC[u_1, u_2 ] \subset  \BC[\chi(F_2)].$$
Hence $x\neq 0$. 

Now assume $\epsilon^2=\pm i$. 
 Note that $M$ can be embedded into $S^3$ since each solid torus $M_i$ can. Sikora \cite{Sikora} showed that when $M$ can be embedded into a homology sphere, the skein module $\cS_\epsilon(M)$, with $\epsilon^2=\pm i$, has a commutative algebra structure such that if $\al, \beta$ are framed knots them $\al \beta =  s(\al,\beta) (\al \cup \beta)$ where $s(\al,\beta)\in \{ \pm 1\}$. 
 Moreover the algebra $\cS_\epsilon(M)$ is also isomorphic to the universal $SL_2(\BC)$-character ring of $M$, and we get a surjective algebra homomorphism $\Omega: \cS_\epsilon(M) \to \BC[\chi(M)]$. Now $\Omega(x_i) = \pm q^{d_i} u_i$, where $d_i\in \BZ$. It follows that 
$$ \Omega(x) = (\pm q^{2d_1} u_1^2-4)(\pm q^{2d_2} u_2^2-4) 
\neq 0 \ \text{in } \BC[u_1, u_2 ] \subset  \BC[\chi(F_2)].$$
Hence $x\neq 0$. This completes the proof of the theorem.
\end{proof}

\subsection{Proofs}\label{sec.proof}
 For integers $k,m \ge 0$ let $v_{k,m}$ and $u_{k,m}$ be the elements defined in Figure \ref{fig:slideaa}, which are patterns, i.e. elements of $\TL_{2(k+m)}$. Here each box stands for the Jones-Wenzl idempotent $f_{k+m}$. A circle enclosing a number $k$ means there are $k$ parallel strands passing the circle.
By convention $u_{0,0}= v_{0,0} =\emptyset$.

\def\slideab{\raisebox{-20pt}{\incleps{2 cm}{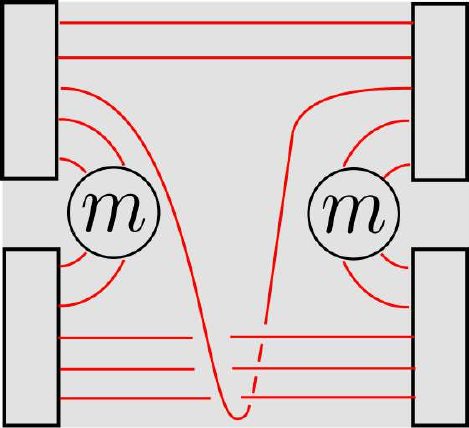}}}
\def\slideabb{\raisebox{-10pt}{\incleps{1 cm}{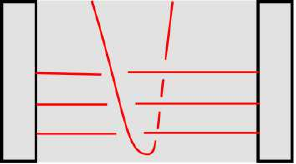}}}
\def\slideabc{\raisebox{-10pt}{\incleps{1 cm}{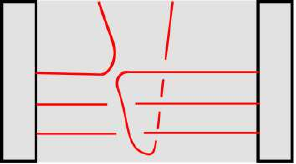}}}
\def\slideabd{\raisebox{-10pt}{\incleps{1 cm}{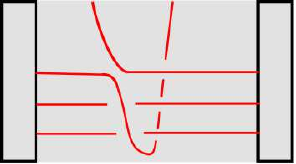}}}
\def\slideabe{\raisebox{-10pt}{\incleps{1 cm}{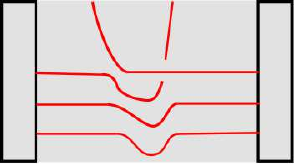}}}
\def\slideabcc{\raisebox{-10pt}{\incleps{1 cm}{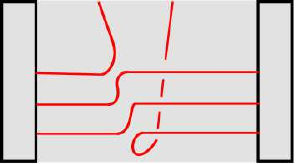}}}
\def\slideabf{\raisebox{-10pt}{\incleps{1 cm}{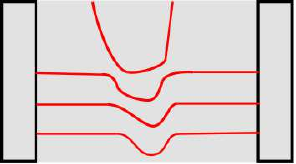}}}
\def\slideabg{\raisebox{-10pt}{\incleps{1 cm}{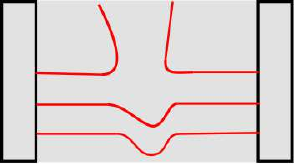}}}
\FIGceps{slideaa}{}{2cm}

The proof of the following main technical lemma uses only the non-returnable property of the Jones-Wenzl idempotent.
\blem\label{r.uvkm} If $k\ge 1$ then
\be
u_{k,m} = q^{4k-2} v_{k,m} + q^{2k-4}(q^{2k } -q^{-2k} ) v_{k-1, m+1}.
\label{eq.uvkm}
\ee
\elem
\begin{proof} The skein relation \eqref{eq.skein0} replaces a crossing with the sum of a positive and a negative resolution of the crossing, each with a positive or negative power of $q$. The non-returning property of the Jones-Wenzl idempotent \eqref{eq.zero} shows that for the upper $2k-2$ crossings in $k_{k,m}$ only the positive resolution results in a non-zero term. Hence 
\be 
u_{k,m} = q^{2k-2} \slideab. 
\label{eq.12}
\ee
Resolve the upper left crossing,
$$ 
\slideabb = q^{} \ \slideabd  + q^{-1} \ \slideabc 
$$
For the first tangle, resolve the crossings on the left from  top to bottom, then the crossings on the right from bottom to top, except for the very last one. Only positive resolutions contribute. For the second tangle resolve the left crossings from top to bottom. Only negative resolutions contribute. Thus we have
$$ 
\slideabb = q^{2k-1} \ \slideabe  + q^{-k} \ \slideabcc
$$
For the left tangle resolve the crossing in two ways. For the right one, note that removing the kink using \eqref{eq.kink}. After that only negative resolutions contribute. Eventually we get
$$
\slideabb = q^{2k} \ \slideabf   +   q^{2k-2 } \ \slideabg -  q^{-2k-2} \ \slideabg.
$$ 
Using the above identity in \eqref{eq.12} we get \eqref{eq.uvkm}.
\end{proof}

\def\eqY{\overset{\#}{=}}
\def\cl{\mathsf{cl}}
\begin{proof}
[Proof of Theorem \ref{thm.connected3}] Assume that the shaded rectangle $D$ in the picture of $v_{k,m}$ (Figure \ref{fig:slideaa}) is embedded in $M= M_1 \#M_2$ so that the separating sphere $S$ of the connected sum $M_1 \#M_2$ meets $D$ in the vertical line separating $D$ into two equal halves.
In what follows  $x \eqY x'$ for $x,x'\in \TL_{2k+2m}$ means  that if $\cl(x)$ and $\cl(x')$ are closures of $x$ and $x'$ respectively by the same closing element not meeting $S $, then $\cl(x)=\cl(x')$ as elements of 
 $\Sq(M)$.

Sliding the top strand of $v_{k, m}$ over the sphere $S^2$ and taking into account the framing, we~get
 $$ u_{k,m} \eqY q^{-6} v_{k,m}, \quad \text{if} \ k\ge 1.  $$
Using \eqref{eq.uvkm}, we get
$$ q^{-6} v_{k,m} \eqY q^{4k-2} v_{k,m} + q^{2k-4}(q^{2k } -q^{-2k} ) v_{k-1, m+1}$$
Multiply by $q^{4-2k}$,
$$ (q^{-2k-2 } -q^{2k+2} )   v _{k,m} \eqY   (q^{2k } -q^{-2k} )   v _{k-1,m+1}  $$
Replacing $k$ by $k-1$ and continue until $k=1$, we get
\be  (q^{-2k-2 } -q^{2k+2} )   v _{k,m} \eqY   (-1)^k(q^{2 } -q^{-2} )   v _{0,m+k}.  
\label{eq.eq1}
\ee
Let $m=0$ and $k=N-1$. 
The scalar of the left hand side is 0 because $\ord(q^4)= N$, and  the scalar of the right side is not 0. Hence
$ v_{0, N-1} \eqY 0.$

Since any element of $F_{q^{1/2}}(M_1) \ot F_{q^{1/2}}(M_2)$ is a linear combination of closures of $v_{0,N-1}$, we have  $F_{q^{1/2}}(M_1) \ot F_{q^{1/2}}(M_2)\eqY 0$. This completes the proof of Theorem \ref{thm.connected3}.
\end{proof}

\no{

\brem The proof showed that when $\ord(q^4)=2$ the image of $\Psi_{\bM_1, \bM_2;q^{1/2}}$. It is not difficult to find $M_1,M_2$ such that $\Sq(M_1 \# M_2)$ is non-zero. For example $M_1=M_2=\BA$. In this case the map $\Psi_{\bM_1, \bM_2;q^{1/2}}$ is not surjective.
\erem
}

\def\Zq{{\BZ[q^{\pm 1/2}]}}

\begin{proof}[Proof of Theorem \ref{thm.zero}] By definition there are components $e_1, e_2$ of $ \cN$ and  a properly embedded path $a$ in $M$ meeting $S$ transversally at one point such that one endpoint of $a$ is in $e_1$ and the other is in $e_2$. It might happen that $e_1=e_2$.
 Let $\alpha$ be a stated $\cN$-link contained in $M\setminus S$.
We can embed the shaded square $D$ into $M\setminus \alpha$ so that the left side of $D$ is $e_1$, the right side is $e_2$, and $S$ meets $D$ in the vertical line dividing $D$ into two equal halves.

For integers $k,m \ge 0$ let $v'_{k,m}$ and $u'_{k,m}$ be the stated diagrams on $D$ as given in Figure \ref{fig:slidexa}.

\FIGceps{slidexa}{}{2.5cm}
Here on a left  side or right side, from bottom to top,  there are $k+m$ negative states followed by $k+m$ positive states. Note the similarity between $v'_{k,m}$ and $v_{k,m}$, and $u'_{k,m}$ and $u_{k,m}$. Instead of the Jones-Wenzl boxes at the boundary in $v_{k,m}$ and $u_{k,m}$ we have states, all positive or all negative in a place where we had a box before. Because they are the same states, we still have the non-returnable property, by the defining relation \eqref{eq.arcs0} of stated skein modules. Since $\alpha$ is contained outside a neighborhood of $S\cup D$, the proof of Lemma \ref{r.uvkm} and  subsequent arguments, where we used only the non-returnable property of the Jones-Wenzl idempotent, are still valid if we replace $v_{k,m}$ and $u_{k,m}$ by $v'_{k,m}\sqcup \alpha$ and $u'_{k,m}\sqcup \alpha$. Thus we have the analog of Identity \eqref{eq.eq1}
\be  (q^{-2k-2 } -q^{2k+2} )   (v' _{k,m}\sqcup \alpha) =   (-1)^k(q^{2 } -q^{-2} )  ( v' _{0,m+k} \sqcup \alpha).
\notag
\ee
Again let $m+k=N-1$. Then the left hand side is 0.  Hence $v' _{0,N-1}\sqcup \alpha=0$. But $v' _{0,N-1}$ consists of $2(N-1)$ trivial arcs, each has one positive and one negative state. From the defining relation \eqref{eq.arcs0} we have
$$ v' _{0,N-1}\sqcup \alpha = q^{l/2} \alpha, \ l \in \BZ.$$
It follows that $\alpha=0$ in $\Sq\MN$.
\end{proof}

\begin{proof}[Proof of Theorem \ref{thm.CH}] 

Let  $M$ be the result of removing the 3-ball $d \times c$ from $S^2 \times S^1$, where $d\subset S^2$ is an open disk and $c\subset S^1$ is an open interval. The embedding $Y \embed S^2 \times S^1$ induces an isomorphism of skein modules. The skein module $\cS(S^2 \times S^1; \Zq)$ over the ring $\Zq$ has been calculated by Hoste and Przytycki  \cite{HP}:
\be 
\cS(S^2 \times S^1; \Zq) = \Zq \oplus  \bigoplus_{i=1}^\infty \Zq/(1-q^{2i+4}),
\label{eq.HP}
\ee
where the first component $\Zq$ is the free $\Zq$-module generated by $\emptyset$. By change of ground ring we have   $\emptyset\neq 0$ in $\Sq(Y)$, for any non zero $q^{1/2}\in \BC$.

Let $E$ be the disk $E=d'\times \{t\}$, where $d'= S^2 \setminus d$ and $t\in c$, and $e$ be an open interval in $E$. Let $(M', \cN')$ be the result of cutting $Y$ along $(E,e)$, with $\cN'= e_1 \cup e_2$, where each $e_i$ is a preimage of $e$. The two components $e_1$ and $e_2$ are separated by the sphere $S^2 \times t'$ where $t'\not \in c$. By Theorem \ref{thm.zero} we have $\emptyset=0$ in $\Sq(M', \cN')$, but $\emptyset$ is not 0 in $\Sq(Y)$. This shows  the cutting homomorphism is not injective.
\end{proof} 

\brem The component $\Zq/(1-q^{2i+4})$ in \eqref{eq.HP}  is generated over $\Zq$ by $x_i$ which is $i$ parallel copies of the a curve $x \times S^1$, where $x\in S^2$. By changing the ground ring to $\BC$ with $q$ a root of 1 with $\ord(q^4) =N > 1$, we see that the element $x_{N-2}$ is not zero in $\Sq(Y)$. However, a calculation can also show that the image of $x_{N-2}$ under the cutting homomorphism is 0. This gives another example of elements in the kernel of the cutting homomorphism.
\erem

\section{Comodule and module structures on stated skein modules }\label{sec:comodule}

\subsection{Marked surfaces}  
\label{sec.basis0}

\begin{definition}[Marked surface]
A marked surface is a pair $(\Sigma, \P)$, where $\Sigma$ is a  compact oriented surface with boundary $\pSi$
and $\cP= \cP^0 \sqcup \cP^1$ such that $\cP^0$  consists of a finite number of signed points in the boundary $\pSi$, called marked points, and $\cP^1$ the union of some  oriented components of $\pSi$ not having marked points.  We also assume each connected component of $\Sigma$ has at least one marked point.

The orientation of a component of $\cP^1$ is positive if it is the one induced from the orientation of $\Sigma$ and negative else.

Each connected component  of $\P^1$ is  called a ``circular marking''. A circle component of $\pS \setminus \P^1$  is called a ``puncture component''.

The thickening of $(\Sigma, \P)$, also called a thickened marked surface, is the
 marked 3-manifold $\MN$ where $M =\Sigma \times (-1,1)$,  and $\cN= (\P^0 \times (-1,1)) \cup \P^1 $, with the identification $\Sigma\equiv \Sigma\times \{0\} \subset M$. The orientation of the boundary edge $p \times (-1,1)$, for $p\in \P^0$, is the positive or the negative orientation of $(-1,1)$ according as $p$ is positive or negative.
 
The stated skein module of $(\Sigma, \P)$ is $\cS(\Sigma, \P)= \cS\MN$.

An embedding of marked surfaces $j:(\Sigma,\P)\hookrightarrow (\Sigma',\P')$ is an orientation preserving  proper  embedding of surfaces $j:\fS\hookrightarrow \fS'$ such that $j(\P)\subset \P'$ and $j$ preserves the orientation of each component of $\cP$.
 \end{definition}

Such an embedding induces in embedding of the corresponding thickened surfaces, and hence an $R$-linear map
$j_*: \cS(\Sigma, \P) \to \cS(\Sigma', \P').$

\brem Our marked surfaces are finer than usual as marked points are signed, and the marking set might contain oriented circular boundary components. For technical reasons we  require that each connected component of $\fS$ has at least one marked point.
\erem

Suppose $\P$ has no circular component. The skein module $\cS(\Sigma, \P)$ has an algebra structure where the product of two  stated $\cN$-tangles $\al$ and $\beta$ is obtained by stacking $\al$ above $\beta$. This means, we first isotope so that $\al \subset \Sigma \times (0,1)$ and $\beta \subset (-1,0)$ then define $\al \beta = \al \cup \beta$.  For $e=p\times(-1,1)$ where $p\in \P^0$  we will denote $\inv_p=\inv_e$, where $\inv_e$ defined in Proposition \ref{prop:inversion}. 
Remark that with this product, the map  $\inv_p$ is an algebra isomorphism.

For a marked surface $(\Sigma,\cP)$ its associated boundary-oriented surface $\fS= \fS_{(\Sigma,\cP) }$ is defined as follows. For each $p\in \P^0$ let $N(p)\subset \pSi$ be a small open interval containing $p$. Let $\fS$ be the result of removing the boundary $\pSi$, except for $\P^1$ and all the $N(p)$, from $\Sigma$:
$$\fS= (\Sigma\setminus \pSi ) \cup \P^1 \cup (\bigcup_{p\in \P^0} N(p)) .$$ 
The orientation $\ori$ on $\pfS= (\bigcup_{p\in \P^0} N(p)) \cup \P^1$ is defined by:  A component in $\cP^1$ is already oriented, while a component $N(p)$ is oriented by the orientation coming from $\fS$ or its reverse according as  $p$ is positive or negative. The resulting $(\fS,\ori)$ is a boundary-oriented surface. 

The requirement that each connected component of $\Sigma$ has at least one marked points implies that each connected component of $\fS$ has a boundary edge. Conversely, it is easy to see that every boundary-oriented surface $\fS$ where each connected component has at least one boundary edge is of the form $\fS= \fS_{(\Sigma, \cP)}$ for certain marked surface $(\Sigma, \cP)$.

Identify $\cS(\Sigma, \P)$ with $\cS(\fS)$ via the following $R$-linear isomorphism
$$ \cS(M, \cN) \xrightarrow{\inv  }   \cS(M, \cN')  \xrightarrow{f_* }   \cS(M, \P^1 \cup (\cup_{p} N(p)  ), $$
where
\begin{itemize}
\item $\inv$ is the composition of  $\inv_p$ of all negative $p\in \P^0$, and $\cN'$ is the same as $\cN$, except that the orientation of each $p \times (-1,1)$ is the one coming from $(-1,1)$, 
\item $f: (M,\cN') \to (M, \P^1 \cup (\cup_{p} N(p)  )$  is identity except in a small neighbourhood of each $N(p) \times (-1,1)$ in which it rotates $p\times (-1,1)$ by $\pi/4$ or $-\pi/4$ to make it become $N(p)$, matching   the natural orientation of $p\times (-1,1)$ with the orientation $\ori$ of $N(p)$.
\end{itemize}
We will often use the above identification $\cS(\Sigma, \P)\equiv \cS(\fS).$ With this identification, the skein module $\cS(\fS,\ori)$ of a circle free boundary-oriented surface  has an algebra structure which was studied in many works, for example \cite{Le:TDEC,CL,LY1,LY2,KQ}.

\def\pPP{\partial \PP}
\def\boP{{\mathbf P}}
\subsection{The bigon} \label{ss.Oq}

From now on let $\PP_2$ be the boundary-oriented bigon where the boundary orientation is positive on one edge, called the right edge $e_r$, and negative on the other, called the left edge $e_l$, see Figure \ref{fig:bigon10}(a). The corresponding marked surface is $\boP_2= (D, \cP)$, where $D$ is the standard closed disk and $\cP$ consists of two points in $\partial D$, one positive and one negative.

\FIGceps{bigon10}{(a) Bigon $\PP_2$. (b) The horizontal arc. (c) Product $xy$}{2.5cm}

 Let $a,b,c,d$ be the stated $\pPP_2$-arc of Figure \ref{fig:bigon10}(b), where $\nu\mu$ are respectively $++, +-, -+, --$. In \cite{Le:TDEC,CL} it is proved that the algebra $\cS(\PP_2)$ is generated by $a,b,c,d$ subject to the following relations:
 \begin{align*}
  ba&= q^2 ab, ca = q^2 ac, db = q^2 bd, dc = q^2 cd \\
 bc &= cb, ad - q^{-2} bc = da - q^{2} bc =1.
\end{align*} 
The product of two elements, represented by stated $\pPP_2$-tangle diagrams $x$ and $y$, is the union $x\cup y$ where we first isotope $x$ so that it is higher than $y$, see Figure \ref{fig:bigon10}(c).

In \cite{CL} we defined geometrically the  coproduct, counit, and antipode which make $\cS(\PP_2)$ a Hopf algebra. The coproduct is particularly simple: by cutting the bigon $\PP_2$ along the ideal arc connecting the two vertices, we get two copies of $\PP_2$, and the cutting homomorphism is the coproduct $\Delta$. On the generators the counit $\ve$ and the antipode $S$ are given by
\begin{align}
\ve(a) &= \ve(d)=1, \ve(b)= \ve(c)=1 \label{eq.epsilon}\\
S(a) &= d, S(d)=a, S(b) = - q^2 b, S(c) = - q^{-2} c.  \label{eq.antipod}
\end{align}
The above Hopf algebra is the well known quantised coordinate algebra $\OSL$ of the Lie group $SL_2(R)$. 
The following was proved in \cite{CL}:
\begin{proposition}
Let $(\fS,\ori)$ be a boundary oriented surface and $e\subset \partial \fS$ a positive (resp. negative) edge. Then $\cSs(\fS,\ori)$ is a right (resp. left) algebra comodule over $\cSs(\PP_2)=\OSL$ with the coaction induced by cutting along an oriented edge $e'\subset \fS$ parallel to $e$. 
\end{proposition}

\subsection{The annulus}
\label{ex:annulus}
Let $\AA= [-1,1] \times S^1$ be the boundary-oriented annulus with one positive orientation on one boundary component, called the right component,  and one negative orientation on the other, called the left component. A slit along a properly embedded  arc connecting the two boundary components yields the bigon $\PP_2$, where the right (resp. left) component goes to the right (resp. left) edge. By Theorem~\ref{thm.slit2}, 
$$ \cS(\AA) = \cS(\PP_2)/\sim = \OSL/\sim.$$
 Relation $\sim$ of Theorem \ref{thm.slit2}, with the product structure as described in Subsection \ref{sec.basis0}, translates to  $xy = yx$ for all $x,y\in \OSL$. Thus we have
  $$\cS(\AA) = \OSL/(xy-yx),$$
  which is known as  the 0-th Hochchild homology ${\rm HH_0}(\OSL)$.
This space was computed in \cite{FT} over $\mathbb{C}$ when $q$ is not a root of unity; its complete structure when working over arbitrary ground ring $\cR$ is unknown to us.
But it is not difficult to show that over the ring $\Zq$ the module ${\rm HH_0}(\OSL)$ 
 contains torsion. For instance it is an easy exercise to show that $(q^2-1)\tau (ab)=0$ but $\tau (ab)\neq 0$, where $\tau: \OSL \onto  {\rm HH_0}(\OSL)$ is the natural projection and $a,b$ are the generators given in Subsection \ref{ss.Oq}.  We also observe for later purposes that if $\gamma$ is  the core of the annulus then by Lemma \ref{lem:circularskeins} we have $\gamma=2[\emptyset]$ in ${\rm HH_0}(\OSL)$.

Note that the product in $\OSL$ does not descend to a product in ${\rm HH_0}(\OSL)$. However the coalgebra structure does descend to ${\rm HH_0}(\OSL)$. For this we need to check that $\sim$ is a coideal. In fact if
 $\Delta(x)=x_{1}\otimes x_2$ and $\Delta(y)=y_{1}\otimes y_2$ (in Sweedler's notation) then  $$\Delta([x,y])=[x_1,y_1]\otimes x_2y_2+y_1x_1\otimes [x_2,y_2]$$ and $\epsilon([x,y])=0$. Here $[x,y]= xy-yx$.

As in \cite{CL}, the existence of the cutting morphism allows to  prove the following:
\begin{proposition}\label{prop:circlecomodule}
Let $(\fS,\ori)$ be a boundary-oriented  surface with a circular marking $c$ oriented positively (resp. negatively). Then $\cSs(\fS)$ is a right (resp. left) comodule over the coalgebra ${\rm HH_0}(\OSL)$  via the coaction given by cutting along an oriented circle $c'\subset \fS$ parallel to $c$. 
\end{proposition} 
\begin{proof}
We prove the statement for $c$ positive, the other case is similar. 
The annulus bounded by $c'$ and $c$ is identified with the standard annulus $\bA$, where $c$ is the right boundary component.  
The cutting morphism $\Theta_{c'}:\cSs(\fS)\to \cSs(\fS\sqcup A)=\cSs(\fS)\otimes_{\cR}{\rm HH_0}(\OSL)$ is coassociative by Identity \eqref{eq.commu5}. Furthermore if $\alpha(\vec\eta,\vec\nu)$ is a disjoint union of parallel arcs embedded in $A$ and connecting the two boundary components with states $\vec{\eta},\vec{\nu}$ then from \eqref{eq.epsilon} we have $\epsilon(\alpha(\vec{\eta},\vec{\nu}))=\delta_{\vec{\eta},\vec{\nu}}$ so that $(Id_{\fS}\otimes \epsilon)\circ \Delta=Id_\fS$. 
\end{proof}

\subsection{Comodule structure of $\cSs(M,\cN)$}
Let $(M,\cN)$ be a marked manifold and $c$ be a component of $\cN$. We will show that associated to $c$ is a comodule structure of $\cSs(M,\cN)$ over $\cSs(\P_2)$ (if $c$ is an arc) or ${\rm HH}_0(\OSL)$ (if $c$ is a circle). 

Suppose first that $c$ is an oriented arc, i.e. $c$ is the image of $(-1,1)$ via a smooth embedding of $[-1,1]$ in $\partial M$. Let us denote $\overline{c}$ the image of $[-1,1]$ via the embedding and let $N(c)$ be a regular neighbourhood of $\overline{c}$ in $M$.

Let $D^2$ be the unit disc in $\mathbb{C}$ and let $\psi_+:N(c) \to D^2\times I$ (resp. $\psi_-$) be an orientation preserving diffeomorphism sending $c$ to $\{+1\}\times (-1,1)$ (resp. $\{-1\}\times (-1,1)$).
Letting $M'=M\setminus int(N(c))$, there exists an orientation preserving diffeomorphism  $\phi:M'\to M$ unique up to isotopy which is the identity out of a neighbourhood of $N(c)$; let $\cN'=\phi^{-1}(\cN)$ and $c'=\phi^{-1}(c)$. 
Endow $N(c)$ with the marking $\cN''=\{\pm 1\}\times (-1,1)$. 
Cutting $N(c)$ out of $M$ is obtained by cutting along a properly embedded disc $D$ containing $c'$ and by Theorem \ref{thm.gluing} we obtain a morphism: 
$$\Cut_{D,c'}:\cSs(M,\cN)\to \cSs(M',\cN')\otimes_\cR\cSs(N(c),\cN'')=\cSs(M,\cN)\otimes_\cR \cSs(N(c),\cN'')$$
where the second equality is induced by $\phi_*\otimes Id_{\cSs(N(c),\cN'')}.$ 
Now observe that $(N(c),\cN'')$ is diffeomorphic to a thickened bigon endowed with two positive markings. In order to get the bigon with a negative and a positive marking (whose stated skein algebra, as recalled in Subsection \ref{ss.Oq}, is canonically $\OSL$) we need to apply one inversion morphism $\inv$. Via the identification $\psi_+$ (resp. $\psi_-$) the image of $c$ is $e_r=\{1\}\times (-1,1)$ (resp. $e_l=\{-1\}\times (-1,1)$). 
Therefore in order to get a right $\OSL$-module structure on $\cSs(M,\cN)$ we define: $$\Delta_R=(Id_{\cSs(M,\cN)}\otimes (\inv_{e_l}\circ (\psi_+)_*))\circ \Cut_{D,c'}\qquad$$
and in order to get a left $\OSL$-module structure we define: $$\Delta_L=((\inv_{e_l}\circ (\psi_-)_*)\otimes \Id_{\cSs(M,\cN)}) \circ \tau\circ \Cut_{D,c'}$$ where $\tau(x\otimes y)=y\otimes x$. 
\begin{proposition}\label{prop:comodule}
$$\Delta_R:\cSs(M,\cN)\to \cSs(M,\cN)\otimes \cSs(D^2\times I,\{1\}\times [-1,1])$$
is a right comodule structure. 
Similarly 
$$\Delta_L:\cSs(M,\cN)\to \cSs(D^2\times I,\{-1\}\times [-1,1])\otimes \cSs(M,\cN)$$
 is a left comodule structure. 
\end{proposition}
\begin{proof}
The proof is similar to the proof of coassociativity of the coaction for the case of boundary oriented surfaces given in \cite{CL}. 
If $D'\subset M$ is another properly embedded disc parallel to $D$ and we let $c''\subset D'$ be an oriented edge parallel to $c'$, then by the commutativity statement of Theorem \ref{thm.gluing}, we get the associativity of the coaction: $$\inv_{d''}\circ \inv_{d'}\circ \Cut_{D',c''}\circ \Cut_{D,c'}=\inv_{d''}\circ \inv_{d'}\circ \Cut_{D,c'}\circ \Cut_{D',c''}$$ where we let $d'$ (resp. $d''$) be the copy of $c'$ (resp. $c''$) contained in the component of  $\Cut_{D,c'}$ (resp $\Cut_{D',c''}$) containing $c$.\end{proof}

If instead $c$ is a circle marking, let $N(c)$ be a regular neighbourhood of $c$ in $M$, diffeomorphic to the thickening of the annulus $\AA$ (see Subsection \ref{ex:annulus}) via an orientation preserving diffeomorphism $\psi_+$ (resp. $\psi_-$) such that $\psi_+(c)$ is the positive (resp. negative) boundary component of $\AA$. Let $\cN''$ be a marking on $N(c)$ given by $\psi_+^{-1}(\partial \AA\times \{0\})$ (resp. $\psi_-^{-1}(\partial \AA\times \{0\})$).

As above, there exists a diffeomorphism $\phi:M'=M\setminus int(N(c))\to M$ which is the identity out of a regular neighbourhood of $N(c)$, unique up to isotopy. Letting $\cN'=\phi^{-1}(\cN)$ and $c'=\phi^{-1}(c)$, we can then identify $\cSs(M',\cN')$ and $\cSs(M,\cN)$ via $\phi_*$.
Let then $D\subset M$ be a properly embedded annulus containing $c'$; applying Theorem \ref{thm.gluing} we  then define
$$\Cut_{D,c'}:\cSs(M,\cN)\to \cSs(M',\cN')\otimes_\cR\cSs(N(c),\cN'')=\cSs(M,\cN)\otimes_\cR \cSs(N(c),\cN'')$$
where the second equality is induced by $\phi_*\otimes Id_{\cSs(N(c),\cN'')}$. 
By Proposition \ref{prop:circlecomodule} we have $\cSs(N(c), \cN'')={\rm HH}_0(\OSL)$. Therefore we get a right ${\rm HH}_0(\OSL)$-comodule structure on $\cSs(M,\cN)$ via:
 $$\Delta_R=(\Id_{\cSs(M,\cN)}\otimes (\psi_+)_*) \circ \Delta$$ 
 or a left ${\rm HH}_0(\OSL)$-comodule structure via:
 $$\Delta_L=((\psi_-)_*\otimes \Id_{\cSs(M,\cN)} ) \circ \tau \circ \Delta$$
 where $\tau(x\otimes y)=y\otimes x$.

\begin{proposition}
$$\Delta_R:\cSs(M,\cN)\to \cSs(M,\cN)\otimes \cSs(\AA\times [-1,1],\partial \AA\times \{0\})$$
is a right comodule structure. 
Similarly 
$$\Delta_L:\cSs(M,\cN)\to \cSs(\AA\times [-1,1],\partial \AA\times \{0\})\otimes \cSs(M,\cN)$$
 is a left comodule structure. 
\end{proposition}
\begin{proof}
If $D'\subset M$ is another properly embedded annulus parallel to $D$ and we let $c''\subset D'$ be its core oriented as $c'$, then by the commutativity statement of Theorem \ref{thm.gluing}, we get $$\Cut_{D',c''}\circ \Cut_{D,c'}=\Cut_{D,c'}\circ \Cut_{D',c''}$$ which proves co-associativity.\end{proof}

\begin{remark}\label{rem:multicomodule}
If $c,c'\in \cN$ are distinct markings in $M$ then the associated comodule structures commute to each other. 
\end{remark}

\subsection{Module structure of $\cS(M,\cN)$}\label{def:bimodules}
Let $(\Sigma,\P)$ be a marked surface and let $\phi:{\Sigma}\to \partial M$ be an embedding; we will say that the sign of $\phi$ is $+1$ if $\phi$ is orientation preserving and $-1$ else. Recall that by hypothesis each edge of $\cN$ is the image of $(-1,1)$ through an embedding of $[-1,1]$ in $\partial M$; therefore we will talk of ``target'' of the edge $c$ (the image of $\{1\}$) and of its source (the image of $\{-1\}$) and we will denote $\overline{c}$ and $\overline{\cN}$ the closures respectively of $c$ and of $\cN$ in $\partial M$. 

Suppose that $\overline{\cN}\subset \partial M$ is such that $\overline{\cN}\cap \phi(\Sigma)=\phi(\P)$ and that for each $p\in \P$, if $c\in \cN$ is the component such that $\overline{c}\cap\phi(\Sigma)=\{p\}$ then $p$ is the target of $c$ if $\sign(p)\sign(\phi)=1$ and it is the source of $c$ if $\sign(p)\sign(\phi)=-1$ (here 
$\sign(p)$ is the sign of the component of $\fS$ containing it).

Then a regular neighbourhood of $\phi(\Sigma)$ in $M$ is diffeomorphic to $(\Sigma\times [-1,1],\P\times[-1,1])$ : let $i:(\Sigma\times [-1,1],\P\times[-1,1])\to (M,\cN)$ the embedding. 
Furthermore there is a diffeomorphism $\psi:M\to M\setminus int(i(\Sigma\times [-1,1]))$ isotopic to the identity of $M$.
 
We can define a left action (resp. a right action) of $\alpha\in  \cS(\Sigma,\P)$ on $m\in \cS(M,\cN)$ as $$\alpha\cdot m:=[i(\alpha) \sqcup \psi(m)]\  \rm{(resp.} \ m\cdot \alpha:=[\psi(m)\sqcup i(\alpha)])$$ 
where $[x]$ denotes the class in $\cSs(M,\cN)$ of the stated tangle $x$. 
The proof of the following proposition is straightforward and left to the reader: 
\begin{proposition}\label{prop:module}
The above defined structure endows $\cS(M,\cN)$ with the structure of a left module (resp. right module) over $\cS(\Sigma,\P)$.
\end{proposition}
Furthermore recall that for each edge $p\in \P$ the algebra $\cSs(\Sigma,\P)$ is also a $\OSL$-comodule algebra; it is not difficult to prove that the above result actually holds in the category of $\OSL$-comodules, namely that for each $e\in \cN$ if $\cSs(M,\cN)$ and $\cSs(\Sigma,\P)$ are endowed with the right $\OSL$-comodule structure associated to $e$ (resp. $e\cap \Sigma$) then it holds :
$$\Delta_e(\alpha\cdot m)=\Delta_e(\alpha)\cdot \Delta_e(m)\  \rm{(resp.}\ \Delta_e(m\cdot \alpha)=\Delta_e(m)\cdot \Delta_e(\alpha) )$$
where in the right hand side of the equalities $\cdot$ stands for the tensor product of the action and of the product in $\OSL$. 
\begin{remark}
If the set of edges of $\cN$ is ordered, then using Remark \ref{rem:multicomodule} one has actually a comodule algebra structure over a suitable tensor power of $\OSL$ depending on the comodule structure defined in $\cS(M,\cN)$. 
\end{remark}

\subsection{Sphere lemma}
Suppose that one component of $\partial M$ is a sphere endowed with a single oriented arc $e\in \cN$, and let $\cS_0(M,\cN)$ be the sub $\cR$-module generated stated skeins represented by arcs not intersecting $e$. Let also $\hat{M}$ be obtained by filling $M$ with a ball $B^3$ along that boundary component and $\hat{\cN}=\cN\setminus e\subset \partial \hat{M}$. Let ${\cR}^{loc}$ be the ring obtained by localising $\cR$ by the multiplicative set generated by $\mathbb{Z}\setminus \{0\}\cup \{1-q^{2n+4},n\geq 1\}$. 
\begin{proposition}[Sphere lemma]\label{prop:sphere}
$$\cR^{loc}\otimes \cS(M,\cN)=\cR^{loc}\otimes \cS_0(M,\cN)=\cR^{loc}\otimes \cS(\hat{M},\hat{\cN}).$$
\end{proposition} 
\begin{proof}
The second equality is clear as isotopy in $\hat{M}$ is the same as isotopy in $M$ for arcs not touching $\partial B^3$, and we note here that it holds over $\cR$.
To prove the first equality we apply the standard sphere trick: if $\alpha$ is a stated skein such that $\alpha\cap \partial B^3$ has $n$ points then we can isotope one strand of $\alpha$ around $\partial B^3$ so to get the equality of Figure \ref{fig:sphere}.
Then applying relations \eqref{eq.skein0} to all the crossings and then \eqref{eq.arcs0} we have the following equality: $\alpha=+q^{6}\cdot q^{2(n-1)} \alpha+ l.o.t.$ where $l.o.t$ stands for a linear combination of skeins whose intersection with $\partial B^3$ has less than $n$ points.
\begin{figure}
\raisebox{-2cm}{\includegraphics[width=5cm]{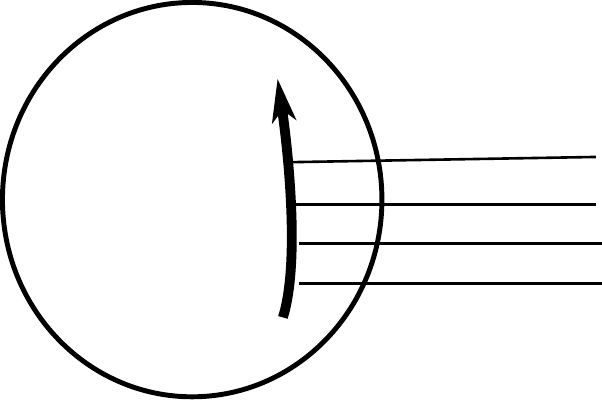}}
$\longleftrightarrow $\raisebox{-2cm}{\includegraphics[width=5cm]{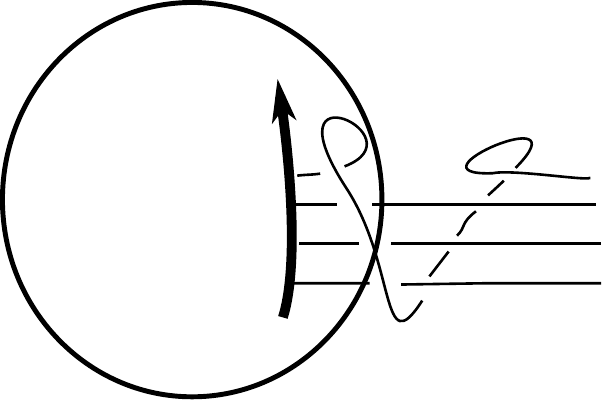}}
\caption{Isotopy of one strand around the sphere in the case $n=4$. }\label{fig:sphere}
\end{figure}
As a consequence since $1-q^{2n+4}$ is invertible in $\cR^{loc}$ we can express $\alpha$ a linear combination of skeins with lower intersection with $B$. 
Arguing by induction on this number of intersections we prove that each skein can be represented as a linear combination of skeins not intersecting $B$. 
\end{proof}

\def\MM{{\mathbf M}}
\def\Sl{{\mathsf{Sl}}}
\def\boSi{{\boldsymbol{\Sigma}}}
\section{Splitting along a strict surface}\label{sec:splitting} 

Suppose $\MM =\MN$ is a marked 3-manifold. A {\em strict subsurface $\Sigma$} of $\MM$ is a proper embedding $\Sigma\embed M$ of a compact surface (so that $\partial \Sigma\subset \partial M$),
$\Sigma$ is traversal to $\cN$ and every connected component of $\Sigma$ intersects $\cN$. Define the slit $\Sl_\Sigma(\MM):= (M', \cN')$, where $M'= M \setminus \Sigma$ and $\cN'= \cN \setminus \Sigma$. For a point $p\in \cP:= \Sigma\cap \cN$ define its sign to be $+$ or $-$ according as  the orientation of $M$ is equal the orientation of $\Sigma$ followed by the orientation of the tangent to $\cN$ at $p$ or not.   
Then $\boSi= (\Sigma, \cP)$ is a marked surface and
there is a right and a left action of $\cS(\boSi)$ on $\cS(\Sl_\Sigma(\MM))$ defined as follows. 
An obvious compactification  $M'$  near $\Sigma$ gives a  manifold $M''$ having two copies $\Sigma_1$ and $\Sigma_2$ of $\Sigma$ on its boundary such that $M''/(\Sigma_1 = \Sigma)$ is $M$.
There is a left action of $\cS(\Sigma_1, \cP)$ on $\cS(M'', \cN')$ and a right action of $\cS(\Sigma_2, \cP)$ on $\cS(M'', \cN')$. Now identify each of $\cS(\Sigma_1, \cP)$ and $\cS(\Sigma_2, \cP)$ with $\cS(\Sigma, \cP)$, and identify $\cS(M'', \cN')$ with $\cS\MN$  via the embedding $(M', \cN') \embed (M'', \cN')$.

The embedding $\Sl_\Sigma(\MM) \embed \MM$ induces an $\cR$-linear homomorphism $\varphi_\Sigma: \cS(\Sl_\Sigma(\MM)) \to \cS(\MM)$.

\bthm
\label{thm.split3} Assume $\Sigma$ is a strict subsurface of a marked 3-manifold $\MM=\MN$. Then $\cS(\MM)= {\rm HH}_0( \cS(\Sl_\Sigma(\MM))  )$, the $0$-th Hochschild homology of the $\cS(\boSi )$-bimodule $\cS(\Sl_\Sigma(\MM))$.

More precisely the $\cR$-linear map $\varphi_\Sigma: \cS(\Sl_\Sigma(\MM)) \to \cS(\MM)$ is surjective and its kernel is the $\cR$-span of $\{ a*x - x*a \mid x\in \cS(\Sl_\Sigma(\MM)), \ a \in \cS(\boSi)\}$.
\ethm
\bpr 
Clearly the map $\varphi_\Sigma$ descends to an $\cR$-linear map $\bar \varphi_\Sigma: HH_0( \cS(\Sl_\Sigma(\MM))  ) \to \cS(\MM)$. We will show that $\bar \varphi_\Sigma$ is bijective.

Let $D$ be a concrete stated $\cN$-tangle transversal to $\Sigma$. An embedding  $\al: (0,1] \embed \Sigma$ is {\em good with respect to} $D$ if $\al(1)\in \cN \cap \Sigma$ and $\al \cap D =D \cap \Sigma$. We are in a situation similar to that in the proof of Theorem \ref{thm.slit1}.
For such a $\al$ define
$ \tilde j_\al(D) \in \cS(\Sl_\Sigma(\MM))$ by the same formula as in Equation \eqref{eq.cup2}, where $\al$ is the horizontal line and $D$ is the red strands.

Let $j_\al(D)= \tau(\tilde j_\al (D))$, where $\tau:\cS(\Sl_\Sigma(\MM)) \onto {\rm HH}_0(  \cS(\Sl_\Sigma(\MM)))$ is the natural projection. Note that
\beq 
\bar \varphi_\Sigma(\tilde j_\al(D))= [D],
\label{eq.51}
\eeq
 where $[D]\in \cS\MN $ is the element represented by $D$. This shows that $\bar \varphi_\Sigma$ is surjective.

 Let $\Sigma'$ be a parallel copy of $\Sigma$ in $M$ and let $\al': (0,1] \embed \Sigma'$ be another embedding which is  good with respect to $D$. 
Let us show that  
$j_\al(D) = j_{\al'}(D)$.
From the definition of the slitting operations, we have
$$\tilde j_{\al'} (  \tilde j_\al(D))= \tilde j_{\al} (  \tilde j_{\al'}(D)) \ \text{in } \ \cS(\Sl_{\Sigma, \Sigma'}(\MM)).$$
Denote the common value of the above by $x$. Note that 
$$\cS(\Sl_{\Sigma, \Sigma'}(\MM))= \cS(\Sl_{\Sigma}(\MM)) \ot \cS(\boSi) = \cS(\boSi) \ot \cS(\Sl_{\Sigma}(\MM)). $$
Using \eqref{eq.51}, we have
$$ j_\al(D) = \tau ( *_r(x)),$$
where $*_r:  \cS(\Sl_{\Sigma}(\MM)) \ot \cS(\boSi)  \to \cS(\Sl_{\Sigma}(\MM))$ is the right action. Similarly
$$ j_{\al'}(D) = \tau ( *_l(x)),$$
$*_l: \cS(\boSi) \ot \cS(\Sl_{\Sigma}(\MM)) \to \cS(\Sl_{\Sigma}(\MM))$ is the left action. Hence, as elements of ${\rm HH}_0( \cS(\Sl_{\Sigma}(\MM)) )$ we have  $j_{\al}(D)=  j_{\al'}(D)$, and we denote this common value by $j(D)$.

Let us show that $j(D)$ depends only on the the isotopy class of $D$. Clearly an isotopy whose support does not intersect $\Sigma$ does not change the values of $j(D)$. In a small neighbourhood of $\Sigma$ an isotopy of $D$ is a finite composition of moves M1 and M2 described in Figure \ref{fig:slit3}, where the horizontal line stands for $\Sigma$. The invariance of $j(D)$ under M1 and M2 was already proved in the proof of  Theorem \ref{thm.slit1}.

All the defining relations of the skein module can be assumed to be away from $\Sigma$. Hence $j: \cS(\MM) \to {\rm HH}_0( \cS(\Sl_{\Sigma}(\MM)) )$ is well-defined. 
By definition $ j \circ \bar \varphi_\Sigma(D) = D$, since if $D$ is a stated $\cN'$-tangle in $M'$ then it does not intersect $\Sigma$. It follows that $\bar \varphi$ is injective, and whence bijective.
\epr

\bexa\label{ex:halfslitrecover} Consider the special case when $\Sigma=D$ is a disk which intersects $\cN$ at one point. In this case Theorem \ref{thm.split3} recovers Theorem \ref{thm.slit1} about the half-ideal splitting of a surface. \eexa

\bexa \label{ex:cptslitrecover}Consider the special case when $\Sigma=D$ is a disk which intersects $\cN$ at two positive points. In this case  Theorem \ref{thm.split3} recovers Theorem \ref{thm.slit2} about the compact spitting of a surface. 
\eexa

\bexa\label{ex:trianglesum}[Triangle sum of marked manifolds]
Let $\boP_3$ be the disc with three marked positive points in its boundary and ${\bf B}$ be its thickening, whose marked edges we denote $e_0,e_1,e_2$.  Let $\Sigma\subset {\bf B}$ be a properly embedded disc intersecting once transversally $e_0$ and such that $e_1$ and $e_2$  are in two distinct connected components of $\Sl_{\Sigma}({\bf B})$. 
Let $\MM'$ be a marked three manifold with at least two edge markings $e'_1,e'_2\in \cN$ and let $\MM$ be obtained by glueing $\MM'$ and ${\bf B}$ by identifying disc neighbourhoods of $e'_i$ and $e_i, i=1,2$ in $\partial \MM'$ and $\partial {\bf B}$. We will say that $\MM$ is obtained by operating a self triangle sum of $\MM'$ (if $\MM'=\MM_1\sqcup \MM_2$ with $e'_i\in \MM_i$ this corresponds to glueing $\MM_1$ and $\MM_2$ to a same ball, whence the name of the operation). 

By Theorem \ref{thm.split3} we get that $\cSs(\MM)=\cSs(\MM')$ as $\cR$-modules: indeed it is sufficient to remark that $\cSs(\Sigma)=\cR$ and that $\Sl_{\Sigma}(\MM)$ is diffeomorphic to $\MM'$. 
\eexa

The conclusion of Example \ref{ex:trianglesum} can be refined by observing that $\cSs(\MM')$ is a right $\OSL^{\otimes 2}$-comodule by the right coaction $\Delta=(\Delta_{1}\otimes Id_{\OSL})\circ  \Delta_2$ where $\Delta_i$ is the right coaction associated to edge $e'_i$ as explained in Proposition \ref{prop:comodule}. 
Then we can endow $\cSs(\MM')$ with the structure of a right $\OSL$-comodule via $\Delta'=(Id_{\cSs(\MM')}\otimes m)\circ \Delta$ where $m:\OSL^{\otimes 2}\to \OSL$ is the product. 
Then the following holds:
\begin{theorem}[Triangle sum of marked manifolds]\label{thm:3Dtrianglesum}
The inclusion of $\MM'$ in $\MM$ induces the following isomorphism of right $\OSL$-comodules:
$$\cSs(\MM)=\cSs(\MM').$$
In particular if $\MM'$ is the disjoint union of marked manifolds $\MM_1$ and $\MM_2$ containing respectively $e'_1$ and $e'_2$ then the following isomorphism of  right $\OSL$-comodules holds:
$$\cSs(\MM)=\cSs(\MM_1)\otimes_\cR \cSs(\MM_2).$$
\end{theorem}
\begin{proof}
We adopt the notation of Example \ref{ex:trianglesum} and,  up to renaming the marked edges of ${\bf B}$, suppose that $e_1$ and the target of $e_0$ (recall that each $e_i$ is oriented) lie in the same component of $\Sl_{\Sigma}({\bf B})$. Since by Example \ref{ex:trianglesum} we already know  that the map $i:\Sl_{\Sigma}(\MM)\hookrightarrow \MM$ induces an isomorphism of $\cR$-modules, we just need to check that it induces a morphism of right $\OSL$-comodules. 
To see this, observe that if a stated skein $\alpha \subset \Sl_{\Sigma}(\MM)$ has $\Delta(\alpha)=\alpha_0\otimes \alpha_1\otimes \alpha_2\in \cSs(\Sl_{\Sigma}(\MM))\otimes \OSL\otimes \OSL$, then $\Delta'(\alpha)=\alpha_0\otimes (\alpha_1\alpha_2)$ (we suppress sums for clarity); on the other side if we let $\Delta_{\MM}$ be the coaction of $\cSs(\MM)$ then we also immediately see graphically that $\Delta_{\MM}(i_*(\alpha))= i_*(\alpha_0)\otimes \alpha_1\alpha_2$ because all the endpoints of the components in $\alpha_1$ are nearer to the target of $e_0$ and hence higher than the endpoints of $\alpha_2$ in the bigon cut out to define $\Delta_{\MM}$.
\end{proof}

\brem For surface case the triangle sum was discussed in \cite{CL}, where no  circular marking was considered, and the proof used an explicit basis of the stated skein module of surfaces. A proof not using basis was given by Higgins \cite{Higgins} for stated $SL_3$-skein algebra of $SL_3$, and was generalised to $SL_n$-skein modules in \cite{LS}. The proof presented in this paper (only for $SL_2$) is new.
\erem

\section{A stated skein TQFT}

In this section we interpret the stated skein module of marked $3$-manifolds as a monoidal functor from a suitable category of ``decorated cobordisms'' to the category of algebras and their bimodules. In all this section $\cR$ is a fixed ring with a distinguished invertible element $q^{\frac{1}{2}}$. 
\subsection{The category of decorated cobordisms}
Given a marked three manifold $(M,\cN)$, recall that by hypothesis each edge $c$ of $\cN$ is the image of $(-1,1)$ through an embedding of $[-1,1]$ in $\partial M$; therefore we will talk of ``target'' of $c$ (the image of $\{1\}$) and of its source (the image of $\{-1\}$) and we will denote $\overline{c}$ and $\overline{\cN}$ the closures respectively of $c$ and of $\cN$ in $\partial M$. 

\begin{definition}\label{def:decomfd}
A decorated manifold is 5-tuple $\BM=(M,\partial^+M,\partial^-M,\partial ^s M,\cN)$ where:
\begin{enumerate}
\item $M$ is a compact oriented three manifold, 
\item $\partial^s M,\partial^\pm M\subset \partial M$ are compact surfaces with boundary with disjoint interior and oriented as induced by the orientation of $M$, such that $$\partial M=\partial^+M\cup \partial^-M\cup \partial^s M, \qquad {\rm and }\qquad \partial^+M\cap \partial^- M=\emptyset.$$
\item $\cN\subset \partial^sM$ is a finite set of oriented arcs or circles, such that each connected component of $(\partial^\pm M,\partial^\pm M\cap \overline{\cN})$ is a marked surface without circular markings. We define the sign of a marked point (i.e. an element of $\overline{\cN}\cap \partial^\epsilon M,\epsilon \in \{\pm\}$) as $\epsilon$ if the orientation of $\cN$ locally points into $\partial^\epsilon M$ and $-\epsilon$ else.
\end{enumerate}

We will say that a decorated cobordism is ``straight'' if each component of $\overline{\cN}$ intersects both $M_-$ and $M_+$ in its endpoints. 
A diffeomorphism of decorated cobordisms is an orientation preserving diffeomorphism preserving all the above structures.

\end{definition}
\begin{figure}
{\includegraphics[width=8cm]{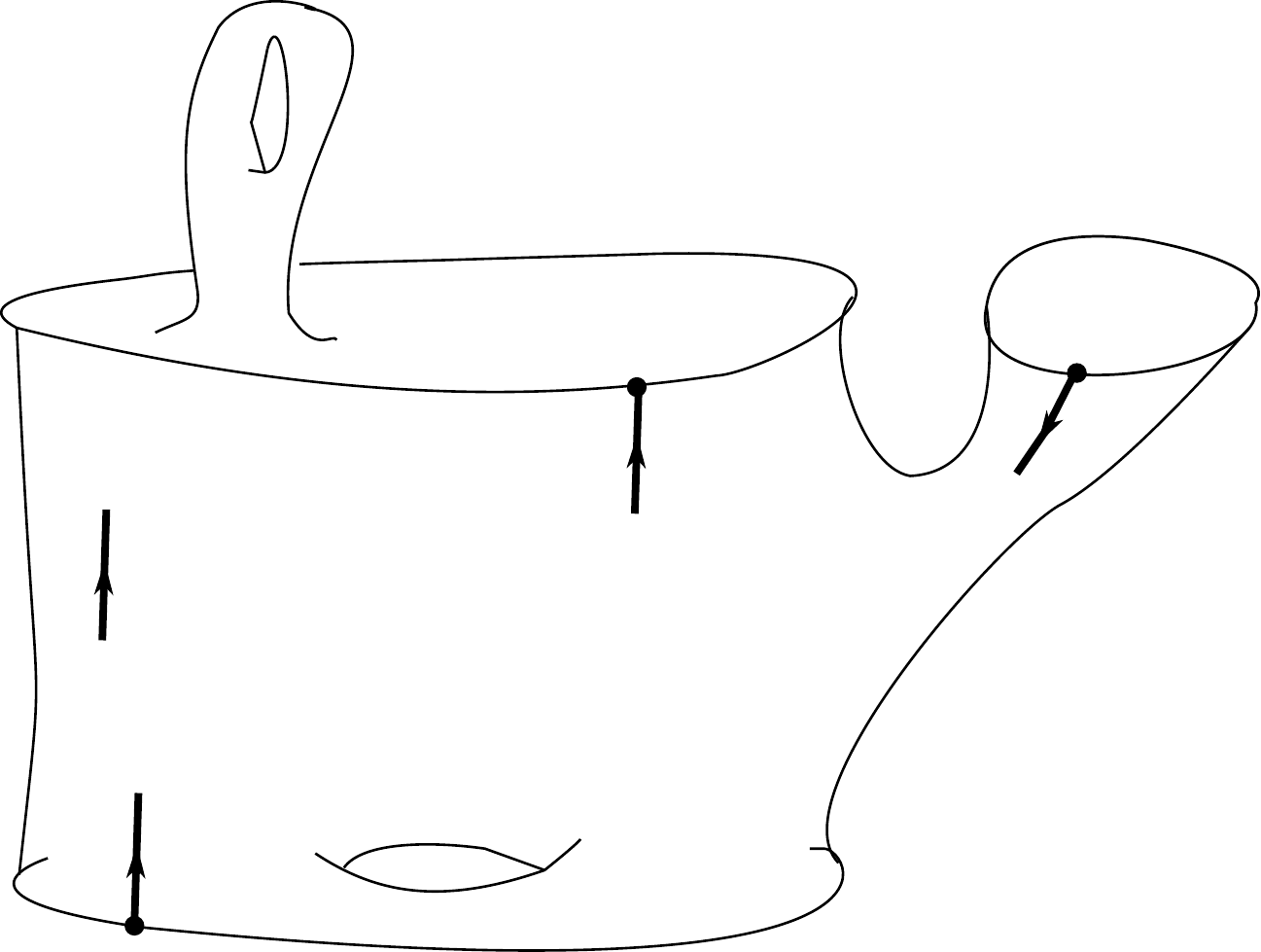}} \put(-205,-5){$+$}\put(-120,106){$+$}\put(-33,106){$-$}\put(-142,115){$\partial_+M$}\put(-90,0){$\partial_-M$}\put(-100,50){$\partial_s M$}
\caption{A decorated cobordism.}
\end{figure}

\begin{remark}\label{rem:nonemptyboundary}
Since the empty is considered to be a marked surface $\partial^\pm M$ can be empty. 
\end{remark}

Associated to each decorated cobordism $\BM$ is an underlying marked three-manifold $(M,\cN)$ and its stated skein module: $\cSs(\BM)=\cSs(M,\cN)$, which is endowed with the natural structure of left module over $\cSs(\partial^+\BM)$ and right module over $\cSs(\partial^-\BM)$.  

\begin{definition}[The category of decorated cobordisms]
$\Dcob$ is the category whose objects are non-empty marked surfaces and $1$-morphisms are described as follows.
A morphism from $\Sigma_-$ to $\Sigma_+$ is the diffeomorphism class of an admissible decorated manifold $\BM$ endowed with diffeomorphisms $\phi_\pm:\partial^\pm \BM\to \Sigma_\pm$ with $\phi_+$ orientation preserving and $\phi_-$ orientation reversing. 
The composition of a morphism $\BM_1:\Sigma_-\to \Sigma$  and $\BM_2:\Sigma\to \Sigma_+$ (with boundary identifications $(\phi_i)_\pm,i=1,2$) is the decorated manifold obtained by glueing $\BM_2$ and $\BM_1$ along $(\phi_2)_-^{-1}\circ (\phi_1)_+$; the arcs of $\overline{\cN}\subset \partial M$ are the images of the arcs of $\overline{\cN}_i,i=1,2$ which do not intersect $\partial^+\BM_1=\partial ^-\BM_2$ and those obtained by glueing the remaining arcs as follows. Let $\overline{c}_1\in \overline{\cN}_1$ be an arc intersecting $\partial ^+\BM_1$ in a point $p$ and let $\overline{c}_2\in \overline{\cN}_2$ be the arc starting from $(\phi_2)_-^{-1}\circ (\phi_1)_+(p)$; by construction the orientations of $\overline{c}_1$ and $\overline{c}_2$ match (they go from $M_1$ to $M_2$ iff $\sign(p)=1$) and thus they define an oriented arc in $\cN\subset \partial^s \BM$.  It can be checked that the so-obtained $\cN$ satisfies the conditions of Definition \ref{def:decomfd}.

\end{definition}

\begin{example}
Let $(\Sigma,\P)$ be a marked surface without circular markings. Then the identity morphism on $(\Sigma,\P)$ is the decorated manifold $\Id_{(\Sigma,\P)}$ with $M=\Sigma\times [-1,1]$, $\partial^\pm M=\Sigma\times\{\pm1\}, \partial^sM=(\partial \Sigma)\times [-1,1]$ and $\cN=\P\times [-1,1]$. 
\end{example}

The category $\Dcob$ is symmetric monoidal with $\otimes$ given by disjoint union. (Actually, as in the standard case of TQFTs, in order to properly define the symmetric monoidal structure one has to consider the category whose objects are surfaces with ordered connected components but we will not detail this point as it is exactly the same as in the standard case.)

Furthermore it is rigid: the dual of a marked surface $\Sigma$ is the surface ${\Sigma}^*$ consisting in $\Sigma$ with the opposite orientation same markings with same signs. The evaluation and co-evaluation morphisms are the morphisms $ev:\Sigma^*\sqcup {\Sigma}\to \emptyset$ and $coev:\emptyset\to \Sigma\sqcup {\Sigma}^*$ represented by the decorated manifold $W=\Sigma\times [0,1]$ with $\partial^-W=\Sigma\times \{0\}\sqcup \Sigma\times \{1\}, \partial^+W=\emptyset$ and $\partial^s W=\overline{\partial W\setminus \partial^-W}$ (respectively $\partial^+W=\Sigma\times \{1\}\sqcup \Sigma\times \{0\}, \partial^-W=\emptyset$ and $\partial^s W=\overline{\partial W\setminus \partial^+W}$). 

In particular the composition $ev_{\Sigma}\circ coev_{{\Sigma}^*}$ (``the quantum trace'') is the decorated manifold $T=\Sigma\times S^1$ with $\partial\pm T=\emptyset$, $\partial^sT=\partial \overline{\Sigma}\times S^1$ and $\cN=\P\times S^1$ where $\P\subset \partial \Sigma$ is the marking of $\Sigma$.

 \subsection{Description of the main theorem}
 If $\cSs(\BM):\Sigma_-\to \Sigma_+$ is a decorated cobordism then $\cSs(\BM)$ is a right module over $\cSs(\Sigma_-)$ and a left module over $\cSs(\Sigma_+)$. 
Let $\Bim$ be the ``Morita category'' whose objects are $\cR$-algebras and morphisms are isomorphism classes of bimodules in the category of $\cR$-modules. The composition is given by the tensor product over the mid algebra (which is well defined up to isomorphism). The identity of an algebra $A$ is the isomorphism class of $A$ as left and right bimodule over itself via left and right multiplication. It is a symmetric monoidal category with the tensor product $\otimes_{\cR}$ and symmetry $A_1\otimes_{\cR} A_2\to A_2\otimes_{\cR} A_1$ for every two algebras $A_1,A_2$. 

It is also rigid with the dual of an algebra $A$ being $A^{op}$ and the left evaluation morphism $A^{op}\otimes_{\cR} A\to \cR$ being the isomorphism class of the bimodule $A$ with natural $\cR$-left module structure and right $A^{op}\otimes_{\cR}A$-module structure given by $a\cdot (a_1\otimes a_2)=a_1aa_2$.
Similarly the left coevaluation is the isomorphism class of the bimodule $A$ seen as a right $\cR$-module and a left $A\otimes A^{op}$-module with action $(a_1\otimes a_2) \cdot a=a_1aa_2$. 

Then the following is the main result of this section: 
\begin{theorem}\label{thm.tqft}
$\cSs:\Dcob\to \Bim$ is a symmetric monoidal functor. 
\end{theorem}
\begin{proof}
Assuming first that $\cSs$ is a functor, its symmetric monoidality is a direct consequence of the fact that $\cSs(\Sigma_1\sqcup \Sigma_2)=\cSs(\Sigma_1)\otimes_\cR \cSs(\Sigma_2)$ and $\cSs(\BM_1\sqcup \BM_2)=\cSs(\BM_1)\otimes_\cR \cSs(\BM_2)$ for every marked surfaces $\Sigma_1,\Sigma_2$ and decorated cobordisms $\BM_1,\BM_2$. 

In order to prove that $\cSs$ is a functor, we have to prove that if $\BM_1:\Sigma_{-1}\to \Sigma$ and $\BM_2:\Sigma\to \Sigma_1$ are decorated cobordisms then $\cSs(\BM_2\circ \BM_1)=\cSs(\BM_2)\otimes_{\cSs(\Sigma)}\cSs(\BM_1)$ as $(\cSs(\Sigma_1),\cSs(\Sigma_{-1}))$-bimodules.

Let $i_1:\BM_1\hookrightarrow \BM_2\circ \BM_1$ and $i_2:\BM_2\hookrightarrow \BM_2\circ \BM_1$ be the natural inclusions. We need to prove that the map 
$$(i_2)_*\otimes (i_1)_*:\cSs(\BM_2)\otimes_{\cR} \cSs(\BM_1)\to \cSs(\BM)$$
factors through an isomorphism of $\cR$-modules
$$\phi_*:\cSs(\BM_2)\otimes_{\cSs(\Sigma)} \cSs(\BM_1)\to \cSs(\BM)$$
which is an isomorphism of $\cSs(\Sigma_-)$ and $\cSs(\Sigma_+)$ bimodules. 
Observing that $\BM_1\sqcup \BM_2=\Sl_{\Sigma}(\BM)$, this is a direct consequence of Theorem \ref{thm.split3}.
The fact that $\phi_*$ is an isomorphism of left $\cSs(\Sigma_1)$-modules (resp. of right $\cSs(\Sigma_{-1})$-modules) is a direct consequence of the definition of the actions as $\Sigma_1$ (resp. $\Sigma_{-1}$) are far from $\Sigma$. 
\end{proof}

\subsection{Immediate corollaries of Theorems \ref{thm.split3} and \ref{thm.tqft}}\label{sub:results}
\def\spi{\mathcal{S}_\pi}

\begin{proposition}\label{prop:HH0}
Let $\Sigma$ be a marked surface. Then $\cSs(\Sigma\times S^1)={\rm HH_0}(\cSs(\Sigma))=\cSs(\Sigma)/\{x\cdot y-y\cdot x\}$ where $\cSs(\Sigma)$ is seen as a left and right module over itself.
\end{proposition}
\begin{proof}
Observe that $\Sl_{\Sigma\times \{1\}}(\Sigma \times S^1)$ is diffeomorphic as a marked manifold to $\Sigma\times [-1,1]$.
Therefore the statement is an immediate corollary of  Theorem \ref{thm.split3}.
\end{proof}

Let $(M,\cN)$ be a marked connected oriented $3$-manifold, and let $\hat{M}=M\setminus D^3$ be the complement of an open ball in $int(M)$. Then  $\partial \hat{M}=S^2\sqcup \partial M$; decompose $S^2$ as $\partial (D^2\times [-1,1])$ and set $\partial^\pm\hat{M}=D^2\times \{\pm 1\}$ and $\partial^s \hat{M}=(\partial D^2)\times [-1,1]\sqcup \partial M$; finally let $\hat{\cN}=\{1\}\times [-1,1]\sqcup \cN$. 
Then $\BM=(\hat{M},\partial^+\hat{M},\partial^-\hat{M},\partial^s\hat{M}, \hat{\cN})$ is a decorated cobordism providing a morphism in $\Dcob$ 
$$\BM:(D,(p,+))\to (D,(p,+))$$
from the disc with one marked point to itself.  

\begin{definition}[$SL_2$-Quantum fundamental group]
Let $\spi(M)=\cSs(\BM)$ as an $\OSL$-comodule with respect to the only edge in $\hat{\cN}\setminus \cN$. 
\end{definition}
The following is straightforward:
\begin{proposition}
$\spi$ is a functor from the category whose objects are oriented connected $3$-manifolds and morphisms are orientation preserving embeddings to the category of $\OSL$-comodules. 
\end{proposition}

The following is then a direct corollary of Theorem \ref{thm.tqft}:
\begin{theorem}[Van Kampen's type theorem]\label{teo:fundgroup}
Let $M_1$ and $M_2$ be two connected, oriented manifolds. Then $$ \spi(M_1\# M_2)=\spi(M_1)\otimes_{\cR} \spi(M_2).$$
\end{theorem}
\begin{proof}
Let $M_0=M_1\#M_2$ and observe that $\BM_2\circ \BM_1=\BM_0$. The statement then follows from Theorem \ref{thm.tqft}. 
\end{proof}

Let now $\mathbf{H}_g^+$ (resp.  $\mathbf{H}_g^-$) be the straight decorated cobordism whose underlying $3$-manifold is a handlebody of genus $g$, $\partial^- \mathbf{H}_g^+=(D^2,+)$ the disc with one marked point (resp. $\partial^+ \mathbf{H}_g^-=(D^2,+)$) , $\partial^s \mathbf{H}_g^\pm$ is a regular neighbourhood of $\partial D^2$ and $\partial^+ \mathbf{H}_g^+=\partial \mathbf{H}_g^+\setminus (\partial^-\mathbf{H}_g^+\sqcup \partial^s\mathbf{H}_g^+)$ (resp. $\partial^- \mathbf{H}_g^-=\partial \mathbf{H}_g^-\setminus (\partial^+\mathbf{H}_g^-\sqcup \partial^s\mathbf{H}_g^-)$).
\begin{theorem}\label{teo:heegaard}
Let $M=H_g\sqcup H'_g$ be  a Heegaard decomposition of a closed oriented $3$-manifold. Then $\spi(M)=\cS(\mathbf{H}^+_g)\otimes_{\cSs(\partial^+ \mathbf{H}^+_g)}\cS(\mathbf{H'}^-_g)$.\end{theorem}

\begin{proof}
This is a direct consequence of Theorem \ref{thm.tqft} as $\mathbf{H}^+_g$ is diffeomorphic to a cobordism from a disc $(D,(p,+))$ to a genus $g$ surface with one boundary component and one marked point on it and $(\mathbf{H}')^-_g$ is a cobordism from the latter surface back to $(D,(p,+))$ and by construction $ \mathbf{H}^-_g\circ \mathbf{H}^+_g=\hat{M}$.
\end{proof}

\end{document}